\documentclass{amsart}

\usepackage{amsmath,amssymb,amsthm,graphicx,pstricks,amscd,amsfonts,latexsym}
\usepackage{xcolor}
\usepackage[all]{xy}
\usepackage[normalem]{ulem}
\usepackage{comment}
\usepackage{lscape}
\usepackage{hyperref}
\usepackage{tikz}
\usetikzlibrary{shapes}
\usetikzlibrary{decorations.pathreplacing,decorations.markings} 
\usepackage{graphicx}
\usepackage{caption} 
\usepackage{float}

\usepackage{tikz,tikz-cd}
\usetikzlibrary{snakes}
\usetikzlibrary{shapes.geometric,positioning}
\usetikzlibrary{arrows,decorations.pathmorphing,decorations.pathreplacing}

\usepackage{epstopdf} 

\theoremstyle{plain} 
\newtheorem{theorem}{Theorem}[section]
\newtheorem{lemma}[theorem]{Lemma}
\newtheorem{proposition}[theorem]{Proposition}
\newtheorem{corollary}[theorem]{Corollary}

\theoremstyle{definition} 
\newtheorem{example}[theorem]{Example}

\newtheorem{definition}[theorem]{Definition}
\newtheorem{remark}[theorem]{Remark}

\newcommand{\End}[1]{\operatorname{\rm End}_{#1}}

\newcommand{\Hom}[1]{\operatorname{{\rm Hom}}_{#1}}

\newcommand{\Ext}[2]{\operatorname{\rm Ext}^{#1}_{#2}}

\newcommand{\add}{\mbox{{\rm add \!}}}

\newcommand{\MOD}{\operatorname{mod}}

\renewcommand{\mod}{\mbox{{\rm mod \!}}}

\newcommand{\proj}{\operatorname{{\rm proj }}}
\newcommand{\perf}{\mbox{{\rm per \!}}}
\newcommand{\Ho}{\mbox{{\rm H}}}

\newcommand{\cB}{\mathcal{B}}
\newcommand{\cC}{\mathcal{C}}

\newcommand{\cG}{\mathcal{G}}

\newcommand{\cM}{\mathcal{M}}

\newcommand{\cP}{\mathcal{P}}

\newcommand{\bP}{\mathbb{P}}

\newcommand{\bS}{\mathbb{S}}
\newcommand{\bZ}{\mathbb{Z}}

\renewcommand{\P}{P^\bullet}

\definecolor{dark-green}{RGB}{14,150,2}
\newcommand{\gpoint}{\color{dark-green}{\circ}}
\newcommand{\rpoint}{\red \bullet}

\newcommand{\LF}{\operatorname{\rm{LF}}}
\newcommand{\MCG}{\operatorname{{\rm MCG}}}
\newcommand{\Diff}{\operatorname{{\rm Diff}}}

\begin{document}
\title[A complete derived invariant for gentle algebras ]{A complete derived invariant for gentle algebras via winding numbers and Arf invariants}

\author{Claire Amiot}
\address{Universit\'e Grenoble Alpes, Institut Fourier, CS 40700, 38058 Grenoble cedex 09}
\email{Claire.Amiot@univ-grenoble-alpes.fr}

\author{Pierre-Guy Plamondon}
\address{Laboratoire de Math\'ematiques d'Orsay, Universit\'e Paris-Sud, CNRS, Universit\'e Paris-Saclay, 91405 Orsay, France}
\email{pierre-guy.plamondon@math.u-psud.fr}

\author{Sibylle Schroll}
\address{Department of Mathematics, University of Leicester, University Road, Leicester LE1 7RH, United Kingdom}
\email{schroll@leicester.ac.uk}

\thanks{The first and second authors are supported by the French ANR grant SC3A (ANR-15-CE40-0004-01).  
The second author is also supported by a French PEPS JCJC grant.
The third author is supported by the EPSRC through an Early Career Fellowship EP/P016294/1.
}

\date{\today}


\begin{abstract}
 Gentle algebras are in bijection with admissible dissections of marked oriented surfaces. In this paper,  we further study the properties of admissible dissections  and we show that silting objects for gentle algebras are given by admissible dissections of the associated surface. We associate to each gentle algebra a line field on the corresponding surface  and prove that the derived equivalence class of the algebra is completely determined by the homotopy class of the line field up to homeomorphism of the surface. Then, based on winding numbers and the Arf invariant of a certain quadratic form over $\mathbb Z_2$, we  translate this to a numerical complete derived invariant for gentle algebras.   
\end{abstract}

\maketitle

\tableofcontents

\section*{Introduction}

Derived categories play an important role in many different areas of mathematics, in particular in  algebra and geometry. However, classifying varieties or algebras up to derived equivalence is, in general, a difficult undertaking. In representation theory, different tools have been developed to handle this problem and have led to tilting theory \cite{tilting}. Even if two derived equivalent algebras share a lot of homological properties, they can be of a very  different nature. Moreover, even for small families of algebras that are  closed under derived equivalence,  it is difficult to establish a complete derived invariant.

In this paper, we give a complete derived invariant for a certain class of algebras called \emph{gentle algebras}. 
6Gentle algebras have been introduced in the early 1980's by Assem and Happel \cite{AssemHappel, AssemHappelErratum} (see also \cite{AssemSkowronski}) and their representation theory is well-studied \cite{ButlerRingel, Krause, Crawley-Boevey}.  Recenlty, gentle algebras have been associated to triangulations or dissections of surfaces, in connection with cluster algebras \cite{AssemBrustleCharbonneauJodoinPlamondon}, with ribbon graphs \cite{Schroll, SchrollCIMPA} or with Fukaya categories of surfaces \cite{HaidenKatzarkovKontsevich}.  It turns out that any gentle algebra can be obtained from a dissection of a surface; this has led to geometric models for their module categories~\cite{BaurCoelho} (building on \cite{CoelhoParsons}) and~$\tau$-tilting theory \cite{PaluPilaudPlamondon2} (see also \cite{BrustleDouvilleMousavandThomasYildirim, PaluPilaudPlamondon}).

From the point of view of homological algebra, the class of gentle algebras is of particular interest, since it is closed under derived equivalence \cite{SchroerZimmermann}.  Their derived categories are well-understood 
\cite{BekkertMerklen, BurbanDrozd2005, ArnesenLakingPauksztello, CanakciPauksztelloSchroll, BurbanDrozd, KalckYang}.  Moreover, in \cite{AvellaAlaminosGeiss}, Avella-Alaminos and Geiss introduced a numerical derived invariant distinguishing between many derived equivalence classes of gentle algebras.  This invariant has sparked a lot of activity on the subject~\cite{Avella-Alaminos, DavidRoeslerSchiffler, AmiotGrimeland, Amiot, Bobinski, Kalck, Nakaoka}.  However, the invariant of Avella-Alaminos and Geiss is not complete, in the sense that it does not distinguish between all derived equivalence classes of gentle algebras.

The derived categories of gentle algebras also enjoy a geometric model enconding their indecomposable objects and the morphisms between them \cite{OpperPlamondonSchroll}.  This model allows for a natural geometric interpretation of the derived invariant of Avella-Alaminos and Geiss.  In this paper, we use this interpretation to refine this invariant into a complete derived invariant for gentle algebras.

\medskip

More precisely, to a gentle algebra, using the geometric description given in \cite{OpperPlamondonSchroll}, one can associate a marked surface $(S,M,P)$ together with a dissection $\Delta$. The data of the dissection $\Delta$ allows to construct a line field $\eta(\Delta^*)$ on the surface $S$. The main result of the paper is then the following.

\begin{theorem}[\ref{theo::derivedInvariants}]\label{theo::intro}
Let $A$ and $A'$ be two gentle algebras associated with dissected surfaces $(S,M,P,\Delta)$ and $(S',M',P',\Delta')$ respectively. Then $A$ and $A'$ are derived equivalent if and only if there exists a homeomorphism of marked surfaces $\Phi:(S,M,P)\to (S',M',P')$ such that the line fields $\Phi^*(\eta(\Delta^*))$ and $\eta(\Delta'^*)$ are homotopic.
\end{theorem}

The idea of associating a line field to a dissected surface comes from
the recent connections of the derived categories of gentle algebras with Fukaya categories of surfaces with boundaries and stops \cite{HaidenKatzarkovKontsevich, LekiliPolishchuk}. In \cite{HaidenKatzarkovKontsevich}, it is shown that the partially wrapped Fukaya category of a surface with stops can be thought of as the derived category of a differential graded gentle algebra (with zero differential). The Fukaya category only depends on the data of the surface with stops and the homotopy class of the line field up to homeomorphism. Hence, one direction of this result can be deduced (in the homologically smooth case) from \cite{HaidenKatzarkovKontsevich}. However, the proof given in the present paper is independent, and does not use the machinery of Fukaya categories. Moreover, it also works for gentle algebras with infinite global dimension, that is, for the non-homologically smooth case. The main ingredient used here is the complete characterisation of silting and tilting objects of the derived category in terms of graded curves (see Theorem~\ref{theo::siltingObjects} and Corollary~\ref{coro::tilting-silting}). This characterisation of silting objects also allows us to give short new proofs of two well-known results on gentle algebras: Namely, of the fact that the class of gentle algebras is closed under derived equivalence (Theorem~\ref{theo::gentle-is-derived-closed}), originally proved in \cite{SchroerZimmermann}, and that gentle algebras are Gorenstein (Theorem~\ref{theo::gentle-is-Gorenstein}), originally proved in \cite{GeissReiten}.

We note that Theorem~\ref{theo::intro} has been independently proved in \cite[Theorem B]{Opper}. Our methods of proof for one of the implications are similar, in that both methods rely on the description of tilting objects. For the other direction, the approaches are different.

However, the main result as stated above is not concretely useful since the computation of the subgroup of homeomorphisms of a marked surface preserving the homotopy class of a given line field is not a realistic task. To make the result more concrete, we use the description of the orbits of homotopy classes of line fields under the action of the mapping class group given in \cite{LekiliPolishchuk}. This allows us to give a numerical derived invariant for gentle algebras which is much easier to compute. This invariant is computed using winding numbers of a basis of the fundamental group of the surface. In the case of gentle algebras coming from surfaces of genus zero, this invariant is precisely equivalent to the Avella-Alaminos-Geiss invariant, while in the case of higher genus, it is a generalisation of it. In genus $\geq 2$ it uses in particular the Arf invariant of some quadratic form over $\mathbb Z/2\mathbb Z$. 

For a dissected marked surface $(S,M,P,\Delta)$ we define simple closed curves $\cG=\{ \alpha_1,\beta_1,\ldots, \alpha_g,\beta_g \} $ and $\cB=\{c_1 ,\ldots,c_{b+p} \} $ as in the following picture:

\[\scalebox{0.8}{\begin{tikzpicture}[scale=0.6, >=stealth]
\shadedraw[top color= blue!30] (-3,0)..controls (-3,1) and (-1,2)..(0,2)..controls (1,2) and (3,1).. (4,1)..controls (5,1) and (7,2).. (8,2)..controls (9,2) and (11,1)..(12,1)..controls (13,1) and (14,2).. (14,3)..controls (14,2.5) and (16,2.5)..(16,3)..controls (16,1) and (18,1)..(18,3)..controls (18,2.5) and (20,2.5).. (20,3)..controls (20,0) and  (18.1,-1)..(18.1,-3).. controls (18.1, -2.9) and (17.9,-2.9).. (17.9,-3).. controls (17.9,0) and (15.1,0).. (15.1,-3).. controls (15.1,-2.9) and (14.9,-2.9).. (14.9,-3).. controls (14.9,-1) and 
 (14,-1)..(12,-1).. controls (11,-1) and (9,-2)..(8,-2).. controls (7,-2) and (5,-1).. (4,-1).. controls (3,-1) and (1,-2).. (0,-2)..controls (-1,-2) and (-3,-1).. (-3,0);

\draw (14,3)..controls (14,3.5) and (16,3.5)..(16,3);
\draw (18,3)..controls (18,3.5) and (20,3.5).. (20,3);

\draw[fill=white] (-1,0)..controls (0,-0.3) and (0,-0.3)..(1,0)..controls (0,0.3) and (0,0.3).. (-1,0);
\draw (-1,0)--(-1.2,0.06);
\draw (1,0)--(1.2,0.06);

\draw[fill=white] (7,0)..controls (8,-0.3) and (8,-0.3)..(9,0)..controls (8,0.3) and (8,0.3).. (7,0);
\draw (7,0)--(6.8,0.06);
\draw (9,0)--(9.2,0.06);

\draw[thick] (0,0) circle (2 and 1);
\draw[thick,->] (2,0.1) --(2,-0.1);
\node at (2.5, 0) {$\alpha_1$};

\draw[thick] (0,-0.2)..controls (0.5,-0.2) and (0.5,-2)..(0,-2);
\draw[thick,<-](0.38,-1.2)--(0.38,-1.4);
\node at (0.8,-1.4) {$\beta_1$};

\draw[thick] (8,0) circle (2 and 1);
\draw[thick,->] (10,0.1) --(10,-0.1);
\node at (10.5, 0) {$\alpha_2$};
\draw[thick] (8,-0.2)..controls (8.5,-0.2) and (8.5,-2)..(8,-2);
\draw[thick,<-](8.38,-1.2)--(8.38,-1.4);
\node at (8.8,-1.4) {$\beta_2$};

\draw[thick] (12,1)..controls (12.5,1) and (12.5,-1)..(12,-1);
\draw[thick,->] (12.4,0.1)--(12.4,-0.1);

\draw[thick] (13.9,2.5)..controls (13.9,2) and (16.1,2)..(16.1,2.5);
\draw[thick,->] (14.9,2.1)--(15.1,2.1);
\node at (15,1.8) {$c_1$};
\node at (15,3) {$\partial_1\Sigma$};

\draw[thick] (17.9,2.5)..controls (17.9,2) and (20,2)..(20,2.5);
\draw[thick,->] (18.9,2.1)--(19.1,2.1);
\node at (19,1.8) {$c_2$};
\node at (19,3) {$\partial_2\Sigma$};

 \draw[thick] (18.1,-3).. controls (18.1,-2.9) and (17.9, -2.9).. (17.9,-3)..controls (17.9,-3.1) and (18.1,-3.1).. (18.1,-3);
  \draw[thick] (15.1,-3).. controls (15.1,-2.9) and (14.9, -2.9).. (14.9,-3)..controls (14.9,-3.1) and (15.1,-3.1).. (15.1,-3);
  
  \node at (18,-1.1) {$c_3$};
  \node at (15,-1.1) {$c_4$};
  
    \draw[thick] (18.5,-1.5)..controls (18.5,-1.3) and (17.6,-1.3)..(17.6,-1.5);
  
  \draw[thick] (15.4,-1.5)..controls (15.4,-1.3) and (14.5,-1.3)..(14.5,-1.5);
 
 \draw[<-](18,-1.35)--(18.3,-1.35);
 \draw[<-](14.9,-1.35)--(15.2,-1.35);

\end{tikzpicture}}\]

It is then enough to compute the winding number with respect to the line field $\eta(\Delta^*)$ (which can be combinatorially  computed using the dissection and its dual) of each curve $\gamma\in \cB\cup \cG$ in order to know the derived equivalence class of the algebra $A$. More precisely we have the following result.

\begin{theorem}[\ref{thm::LP+main}]

Let $A$ and $A'$ be two gentle algebras with associated dissected surfaces  $(S,M,P, \Delta)$ and $(S',M', P', \Delta')$, respectively. Let $\cG=\{\alpha_1,\ldots, \beta_g\}$, $\cB=\{c_1,\ldots,c_{b+p}\}$ (resp. $\cG'=\{\alpha'_1,\ldots, \beta'_{g'}\}$, $\cB'=\{c'_1,\ldots,c'_{b'+p'}\}$ ) subsets of simple closed curves on $S\setminus P$ (resp. $S'\setminus P'$) as before. Then the algebras $A$ and $A'$ are derived equivalent if and only if the following numbers coincide:

\begin{enumerate}

\item $g=g'$, $b=b'$, $\sharp M=\sharp M'$, $\sharp P=\sharp P'$; 

\item there exists a permutation $\sigma\in \mathfrak{S}_{b+p}$ such that $n(\sigma (j))=n'(j)$ and $w^{\Delta^*}(c_{\sigma(j)})=w^{\Delta'^*}(c'_{j})$, for any $j=1,\ldots, b$;
 
 \item for $g=g'\geq 1$ one of the following holds 
 
 \begin{enumerate}
 
 \item for $g=g'=1$,  we have
 
$$\gcd \{w^{\Delta^*}(\gamma),w^{\Delta^*}(c)+2,\gamma\in \cG, c\in \cB\}=\gcd \{w^{\Delta'^*}(\gamma'),w^{\Delta'^*}(c')+2,\gamma'\in \cG', c'\in \cB'\}$$

\item for $g=g'\geq 2$ one the following occurs:
\begin{enumerate}

\item there exist $\gamma\in \cG\cup\cB$ and $\gamma'\in \cG'\cup\cB'$  such that $w^{\Delta^*}(\gamma)$ and $w^{\Delta'^*}(\gamma')$ are odd, or 

\item for any $\gamma\in \cG\cup\cB$ and $\gamma'\in \cG'\cup\cB'$, the numbers $w^{\Delta^*}(\gamma)$ and $w^{\Delta'^*}(\gamma')$ are even and there exists an $i$ with $w^{\Delta^*}(c_i)=0 \ \mod 4$, or

\item for any $\gamma\in \cG\cup\cB$ and $\gamma'\in \cG'\cup\cB'$, the numbers $w^{\Delta^*}(\gamma)$ and $w^{\Delta'^*}(\gamma')$ are even and, for any $i=1,\ldots ,b+p$ we have $w^{\Delta^*}(c_i)=2 \ \mod 4$ and 
$$\sum_{i=1}^g(\frac{1}{2}w^{\Delta^*}(\alpha_i)+1)(\frac{1}{2}w^{\Delta^*}(\beta_i)+1)=\sum_{i=1}^g(\frac{1}{2}w^{\Delta'^*}(\alpha'_i)+1)(\frac{1}{2}w^{\Delta'^*}(\beta'_i)+1) \quad \mod 2$$

\end{enumerate}
\end{enumerate}

\end{enumerate}
\end{theorem}

Therefore the invariant is easily computable once we have a good description of the generators of the fundamental groups of the surface associated to the algebras.

\medskip

The plan of the paper is the following. In Section \ref{sect::section1} we recall several basic geometric definitions of line fields and winding numbers, and explain the construction of the line field $\eta(\Delta^*)$. In Section \ref{sect::admissible-dissections-gentle-algebras}, we recall the results of \cite{OpperPlamondonSchroll} that are used in the paper. The description of silting  and tilting objects is done in Section \ref{sect::siltingObjects}, while the main theorem is proved in Section \ref{sect::derivedInvariants}. The concrete criterion using \cite{LekiliPolishchuk} is explained in Section \ref{sect::MCG}. In Section \ref{sect::reprovingKnownResults}, we use the geometric description to reprove some well-known results on gentle algebras. Examples are presented in Section~\ref{sect::examples}. Finally, Section~\ref{sect::surface-cut-algebras} is dedicated to the special case of certain gentle algebras of global dimension~$2$. These algebras, called surface cut algebras, have another, slightly different, geometric model coming from cluster combinatorics, therefore it may be useful to explicitly translate the new invariants in terms of this other model.

\section*{Conventions}
In this paper, all algebras will be assumed to be over a base field $k$.  
All modules over such algebras will be assumed to be finite-dimensional left modules.
Arrows in a quiver are composed from left to right as follows: for arrows $a$ and $b$ we write $ab$ for the path from the source of $a$ to the target of $b$.
Maps are composed from right to left, that is if $f: X \to Y$ and $g: Y \to Z$ then $gf : X \to Z$.

All surfaces with boundary and punctures  in the paper are considered as open surfaces. They are defined by removing closed discs and points from a compact surfaces.  As such, a surface with boundary~$\Sigma$ does not contain its boundary; however, by construction, the boundary exists on a surface containing~$\Sigma$.

\section{Line fields and admissible dissections}\label{sect::section1}

\subsection{Line fields and winding numbers}\label{subs::line-fields-and-winding}
Most of the material of this section is classical geometry. We recall the definitions and basic properties for the convenience of the reader. We refer to \cite{Chillingworth}.

Let $\Sigma$ be a smooth oriented open surface of genus $g$ with $b\neq 0$ boundary components and~$p$ punctures (that is to say,~$\Sigma$ is obtained by removing~$b$ disjoint closed discs and~$p$ distinct points from a compact surface of genus~$g$). We denote by $T\Sigma$ its tangent bundle.

\begin{definition} A \emph{line field} $\eta$ on $\Sigma$ is a continuous section of the projectivized tangent bundle. So it is a continuous map $\eta:\Sigma\to \mathbb P (T\Sigma)$ such that for any $x\in \Sigma$, $\eta (x)$ is in $\mathbb {P}(T_x\Sigma)$. 
\end{definition}

Note that any vector field on $\Sigma$ (i.e. a continuous section of the tangent bundle) yields a line field, but not all line fields  come from vector fields.

For $x \in \Sigma$, define the map~$D:\bP(T\Sigma)\to \bP(T\Sigma)$ by~$D(x,\ell)= (x, \ell^\perp)$,
where~$\ell^\perp$ is the (unique) line orthogonal to~$\ell$ in~$T_x\Sigma$. 
Note that~$D$ is smooth.

Let $\eta$ be a line field on $\Sigma$. Fix $x_0\in \Sigma$ and $v_0\in T_{x_0}\Sigma$ such that $[v_0]=D\circ\eta(x_0)$. Let $f$ be a $\mathcal{C}^1$-map from $\mathbb S^1\subset \mathbb C$ to $\Sigma$ such that $f(1)=x_0$ and $T_{1}f(1)=v_0$. In what follows, we give a definition of the winding number of $f$ relative to the line field $\eta$.  This mainly follows \cite{Chillingworth}, but here the definition is given for line fields instead of vector fields.

The fiber above $x_0$ of the projection $p:\mathbb{P}(T\Sigma)\to \Sigma$ is a circle so we get the following long exact sequence:

\[ \xymatrix{\pi_2(\Sigma,x_0)\ar[r] &\pi_1(\mathbb{S}^1,1)\ar[r]^-{\iota_*} 
& \pi_1(\mathbb{P}(T\Sigma),[v_0])\ar[r]^-{p_*} & \pi_1(\Sigma,x_0)\ar[r] & \pi_0(\mathbb{S}^1,1)=1 }\]

The universal cover $\widetilde{\Sigma}$ of $\Sigma$ is contractible (it is a disk if $g\geq 1$ or $g=0, b\geq 3$ and the plane for $g=0$ and $b=2$) and so $\pi_2(\widetilde{\Sigma},\widetilde{x_0}) = 1$. Thus the isomorphism $\pi_2(\widetilde{\Sigma},\widetilde{x_0})\simeq \pi_2(\Sigma,x_0)$ implies that $\iota_*$ is injective.

Associated to $f$ we define $Z^f\in \pi_1(\mathbb{P}(T\Sigma),[v_0])$ as $$Z^f(z):=[T_zf(1)] \quad\forall z\in \mathbb{S}^1.$$ In other words, for any $z\in \mathbb{S}^1$ the element $Z^f(z)$ is  the tangent line to the curve $f$ at the point $f(z)$. Since $Z^f(z)$ is in $\mathbb{P}(T_{f(z)}\Sigma)$, we have the equality $p_*\{Z^f\}=\{f\}$ in $\pi_1(\Sigma,x_0)$. 

Associated to $f$ and $\eta$ we define $X^{f,\eta}\in \pi_1(\mathbb{P}(T\Sigma),[v_0])$ as 
$$X^{f,\eta}(z):=D\circ\eta\circ f(z)\quad\forall z\in \mathbb{S}^1.$$
Since $\eta\circ f(z)\in \mathbb{P}(T_{f(z)}\Sigma)$, we also have the equality $p_*(\{X^{f,\eta}\})=\{f\}$. 

Hence the element $\{Z^f\}^{-1}\{X^{f,\eta}\}$ is in the kernel of $p_*$, and so has a unique predecessor in $\pi_1(\mathbb{S}^1,1)$. The orientation of $\Sigma$ induces an orientation of each tangent space, hence gives a basis element $e$ of $\pi_1(\mathbb{S}^1,1)$ and a bijection $\pi_1(\mathbb S^1,1)\simeq \mathbb Z$.  This leads to the following.

\begin{definition}\label{defwinding}
The \emph{winding number} $w_{\eta}(f)$ is the unique integer such that $$\{Z^f\}^{-1}\{X^{f,\eta}\}=w_{\eta}(f)\cdot e.$$
\end{definition}

\begin{remark}
\begin{enumerate}

\item It is more common to define the winding number for a curve starting tangentially  to the line field instead of normally to the line field as defined here. We note that  extending the definition to any curve, the two definitions coincide and in this paper, it will be more convenient to consider curves starting and ending normally to the line field.
\item The winding number computes the number of U-turns the $\eta$ line makes relatively to the tangent field of $f$.

\item In the case where the line field $\eta$ comes from a vector field $X$, we have the equality $w_{\eta}(f)=-2\omega_X(f)$ where $\omega_X$ is the winding number defined in \cite{Chillingworth}. 
\end{enumerate}
\end{remark}

The following is proved in \cite{Chillingworth} in the case of a vector field, and can easily be generalized to the case of a line field.
\begin{proposition}\label{propLF}
\begin{enumerate}
\item The map $w_\eta$ factors through a map $\pi_1^{\rm free}(\Sigma)\to \mathbb{Z}$.
      This map sends the free homotopy class of any curve~$\gamma$ that has no contractible loops  to~$w_\eta(\gamma)$.

\item Two line fields $\eta$ and $\eta'$ are homotopic if and only if for any $\gamma\in \pi_1^{\rm free}(\Sigma)$, we have $w_\eta(\gamma)=w_{\eta'}(\gamma).$

\item A line field comes from a vector field if and only if for each $\gamma\in \pi_1^{\rm free}(\Sigma)$, the winding number $w_{\eta}(\gamma)$ is even. 

\item The map $w_\eta$ factors through an element of  $\Ho^1(\Sigma,\mathbb{Z}/2\mathbb{Z})$.

\item Denote by $\LF(\Sigma)$ the set of homotopy classes of line fields on $\Sigma$. Then the map $\Phi: \LF(\Sigma)\times \LF(\Sigma)\to \Hom{\rm Set} (\pi_1^{\rm free}(\Sigma),\mathbb Z)$ defined by $\Phi(\eta,\eta'):=w_{\eta}-w_{\eta'}$ factors through a map 
$$\Phi: \LF(\Sigma)\times \LF(\Sigma) \to \Hom{\mathbb{Z}}(\Ho_1(\Sigma,\mathbb{Z}),\mathbb{Z})
=\Ho^1(\Sigma,\mathbb{Z})$$ making $\LF(\Sigma)$ a $\Ho^1(\Sigma,\mathbb{Z})$-affine space.
\end{enumerate}
\end{proposition}

We need to extend the definition of winding number to non-closed curves.
Let~$\gamma$ be a smooth map from the open interval~$(0,1)$ to~$\Sigma$
such that, for~$x$ sufficiently close to~$0$ or~$1$, the tangent line at~$\gamma(x)$ is orthogonal to~$\eta(\gamma(x))$.
Let~$s,t\in (0,1)$ be such that~$\gamma$ is orthogonal to~$\eta$ on~$(0,s]$ and~$[t,1)$.
As before, define
\[
 Z^\gamma(z):=[T_z\gamma(1)] \quad\forall z\in [s,t]
\]
and
\[
 X^{\gamma,\eta}(z):=D\circ\eta\circ \gamma(z)\quad\forall z\in [s,t].
\]
Then~$\{(Z^\gamma)^{-1}X^{\gamma,\eta}\}$ is a well-defined element of~$\pi_1(\bP(T\Sigma), D\circ \eta(\gamma(s)))$.
It is in the kernel of~$p_*$, and hence has a unique predecessor in~$\pi_1(\bS^1, 1) \cong \bZ\cdot e$.
\begin{definition}\label{defi::windingOfOpenCurve}
 The \emph{winding number~$w_\eta(\gamma)$} is the unique integer such that
 \[
  \{(Z^\gamma)^{-1}X^{\gamma,\eta}\} = w_\eta(\gamma) \cdot e.
 \]

\end{definition}

The following is an easy consequence of the definition.
\begin{proposition}\label{prop::compose winding}
Let $\gamma_1$ and $\gamma_2$ be two smooth maps from~$(0,1)$ to the surface.  Assume that there are~$t_1, t_2\in (0,1)$ such that~$\gamma_1|_{(t_1, 1)}$ and~$\gamma_2|_{(0,t_2)}$ coincide and are orthogonal to the line field~$\eta$.

Then the concatenation~$\gamma$ of~$\gamma_1$ and~$\gamma_2$ is a well-defined smooth curve, and $w_\eta(\gamma)=w_{\eta}(\gamma_1)+w_{\eta}(\gamma_2).$
\end{proposition}

\subsection{Admissible dissections}

In the following, we define a particular type of marked surface with two distinct types of marked points, $\gpoint$-points and $\rpoint$-points. Going forward, we  will always denote such a surface  by $S$. If we consider a marked surface in general, as we did in Section~\ref{subs::line-fields-and-winding}, then we will denote it by $\Sigma$. 

\begin{definition}\label{defi::markedSurface}
 A \emph{marked surface} is a triple~$(S,M,P)$, where
\begin{itemize}
 \item $S$ is an oriented open smooth surface whose boundary is denoted by~$\partial S$; 
 \item $M = M_{\gpoint} \cup M_{\rpoint}$ is a finite set of marked points on~$\partial S$.  
       The elements of~$M_{\gpoint}$ and~$M_{\rpoint}$ will be represented by symbols~$\gpoint$ and~$\rpoint$, respectively.
       They are required to alternate on each connected component of~$\partial S$, and each such component is required to contain at least one marked point;
 \item $P = P_{\gpoint} \cup P_{\rpoint}$ is a finite set of marked points in~$S$, called \emph{punctures}.
       The elements of~$P_{\gpoint}$ and~$P_{\rpoint}$ will be represented by symbols~$\gpoint$ and~$\rpoint$, respectively,
\end{itemize}
If the surface has empty boundary, then we require that both~$P_{\gpoint}$ and~$P_{\rpoint}$ are non-empty.
\end{definition}

\begin{definition}\label{defi::arcs}
 A \emph{$\gpoint$-arc} (or \emph{$\rpoint$-arc}) is a smooth map~$\gamma$ from the open interval~$(0,1)$
 to~$S\setminus P$   such that its~\emph{endpoints}~$\lim_{x\to 0}\gamma(x)$ and~$\lim_{x\to 1}\gamma(x)$
 are in~$M_{\gpoint}\cup P_{\gpoint}$ (or in~$M_{\rpoint}\cup P_{\rpoint}$, respectively).
 The curve~$\gamma$ is required not to be contractible (at the limit) to a point in~$M_{\gpoint}\cup P_{\gpoint}$
 (or~$M_{\rpoint}\cup P_{\rpoint}$, respectively).
\end{definition}

We will usually consider arcs up to homotopy or isotopy.  
Two arcs are said to intersect if any choice of homotopic representatives intersect.

\begin{definition}\label{defi::admissibleDissection}
 A collection of pairwise non-intersecting and pairwise different $\gpoint$-arcs~$\{\gamma_1, \ldots, \gamma_r\}$ on the surface~$(S,M,P)$ is \emph{admissible} if 
 the arcs~$\gamma_1, \ldots, \gamma_r$ do not enclose a subsurface containing no punctures of~$P_{\rpoint}$ and
 with no boundary segment on its boundary.  
 A maximal admissible collection of~$\gpoint$-arcs is an \emph{admissible $\gpoint$-dissection}.
 
 The notion of~\emph{admissible $\rpoint$-dissection} is defined in a similar way. 
\end{definition}

\begin{example}\label{exam::disection}
The following is an admissible~$\gpoint$-dissection of a disc with one puncture~$\gpoint$ and two punctures~$\rpoint$.
\begin{center}
 \begin{tikzpicture}
  \draw[thick] (0,0) circle (2);
  
  \draw[dark-green] (0,0) -- (0,2);
  \draw[dark-green] (0,0) -- (150:2);
  \draw[dark-green] (0,0) -- (-150:2);
  \draw[dark-green] (0,0) -- (-30:2);
  \draw[dark-green] (150:2) -- (-150:2);
  \draw[dark-green] (30:2) -- (-30:2);
  \draw[dark-green] (30:2) -- (90:2);
  \draw[dark-green] (-90:2) -- (-150:2);
  
  \foreach \i in {0,...,5}
  {
  \filldraw[white] ({60*\i+30}:2) circle (0.1);
  \draw[thick, dark-green] ({60*\i+30}:2) circle (0.1);
  }
  \filldraw[white] (0,0) circle (0.1);
  \draw[thick, dark-green] (0,0) circle (0.1);
  
  \foreach \i in {0,...,5}
  {
  \filldraw[red] ({60*\i}:2) circle (0.1);
  }
  \filldraw[red] (-1,0) circle (0.1);
  \filldraw[red] (1,0) circle (0.1);
 \end{tikzpicture}
\end{center}
\end{example}

\begin{proposition}\label{prop::numberOfArcs}
 Let~$(S,M,P)$ be a marked surface, with~$S$ a surface of genus~$g$ such that~$\partial S$ has~$b$ connected components. 
 Then an admissible collection of~$\gpoint$-arcs is an admissible~$\gpoint$-dissection if and only if it contains exactly
 \( |M_{\gpoint}|+|P|+b+2g-2 \) arcs.  
\end{proposition}
\begin{proof}
 We prove the result by induction on the number~$|M_{\gpoint}|+|P|+b+2g-2$. 
 Note that, by our assumptions on the surface (especially the excluded cases in Definition~\ref{defi::markedSurface}),
 we have that~$|M_{\gpoint}|+|P|+b+2g-2 \geq 0$.
 
 Assume that~$|M_{\gpoint}|+|P|+b+2g-2 = 0$. Then~$g=0$, and either~$b=|M_{\gpoint}|=1$ and~$|P|=0$ or~$b=|M_{\gpoint}|=0$ and~$P=2$.
 In the first case,~$(S,M,P)$ is an unpunctured disc with one marked point~$\gpoint$ and one marked point~$\rpoint$ on its boundary; 
 in the second case,~$(S,M,P)$ is a sphere with empty boundary and one puncture~$\gpoint$ and one puncture~$\rpoint$.
 In both cases, the only admissible collection of~$\gpoint$-arcs is the empty collection. 
 Thus the only admissible~$\gpoint$-dissection is the empty set, and the proposition is true in this case.
 
 Assume now that~$|M_{\gpoint}|+|P|+b+2g-2$ is an integer~$n>0$.   
 Then there exists at least one non-empty admissible collection of~$\gpoint$-arcs on~$(S,M,P)$. 
 Indeed, if~$g>0$, then there exists a non-separating curve starting and ending at a point~$\gpoint$;
 this curve forms an admissible collection of~$\gpoint$-arcs. 
 If~$b>1$, then let~$\gpoint$ be on one of the boundary components.
 The arc starting and ending at~$\gpoint$ which follows the boundary component forms an admissible collection of~$\gpoint$-arcs.
 Finally, if~$g=0$ and~$b=1$, then the surface is a disc with~$|M_{\gpoint}|+|P| \geq 2$.
 If~$|M_{\gpoint}| \geq 2$, then any~$\gpoint$-arc between two distinct points of~$M_{\gpoint}$ forms an admissible collection of~$\gpoint$-arcs.
 If~$|M_{\gpoint}| =1$, then~$|P|\geq 1$.
 If~$|P_{\gpoint}| \geq 1$, then any~$\gpoint$-arc between a point of~$M_{\gpoint}$ and one of~$P_{\gpoint}$ forms an admissible collection of~$\gpoint$-arcs.
 If~$|P_{\rpoint}| \geq 1$, then any~$\gpoint$-arc starting and ending at the point of~$M_{\gpoint}$ 
 and enclosing a point of~$P_{\rpoint}$ forms an admissible collection of~$\gpoint$-arcs.
 
 Let~$\Delta$ be a non-empty admissible collection of~$\gpoint$-arcs.
 Let~$\gamma$ be any arc in~$\Delta$.
 Let~$(S',M',P')$ be the topological quotient of~$(S,M,P)$ obtained by identifying all the points in the image of~$\gamma$.
 We treat several cases.
 
 \smallskip
 
 \noindent \emph{Case 1:~$\gamma$ is a non-separating curve starting and ending on the same boundary component.}  
 In this case,~$(S',M',P')$ is a marked surface with~$|M_{\gpoint}|$ marked points~$\gpoint$,~$|P|$ punctures,
 ~$b+1$ boundary components, and genus~$g-1$.  Induction applies. We illustrate this case in an example, the other cases for non-separating curves being similar.
 
 \begin{figure}[H]
\flushleft \includegraphics[width=0.37\textwidth, height=45px, angle=0]{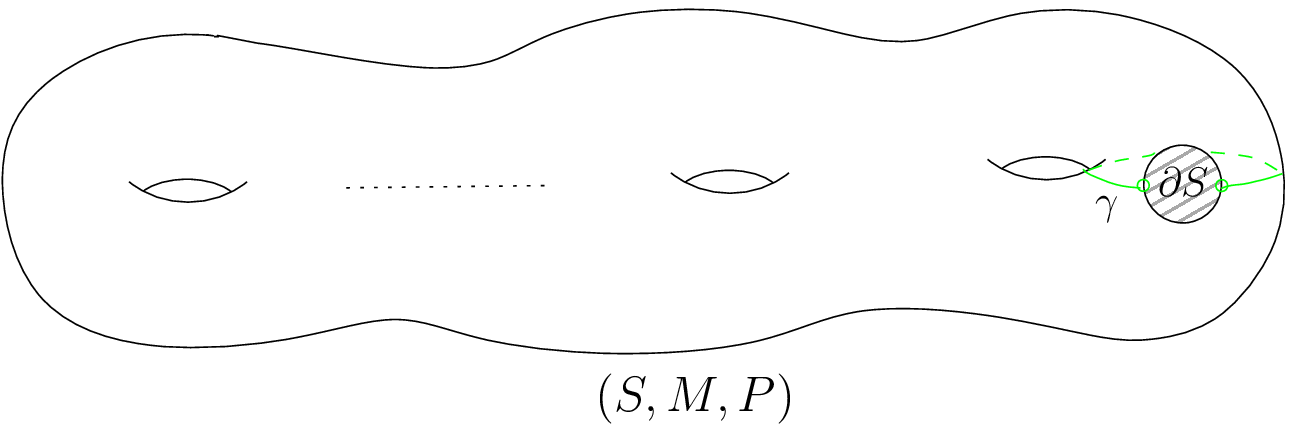}
 
  \vspace{-.1cm}

\centering
  \includegraphics[width=0.37\textwidth, height=56px, angle=0]{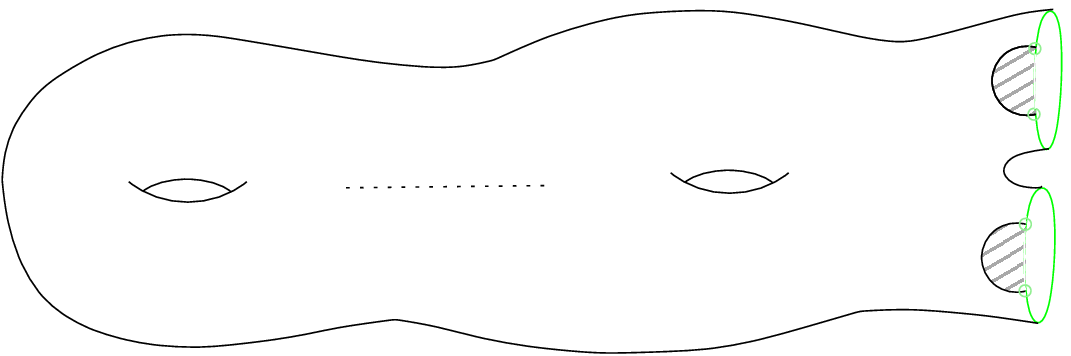}
 
 \vspace{.3cm}
 
 \flushright
 \includegraphics[width=0.37\textwidth, height=45px, angle=0]{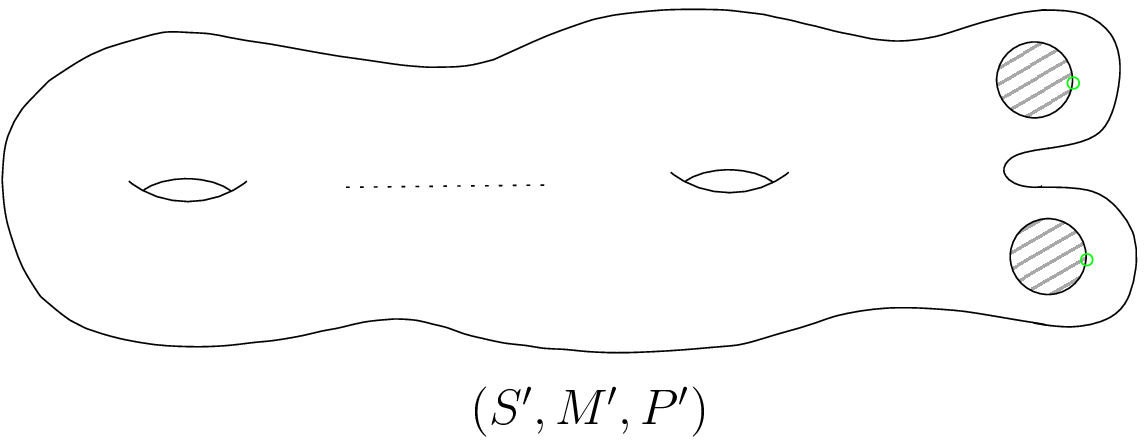}

 \caption*{Example of quotient $(S',M',P')$ of $(S,M,P)$ by a non-separating curve where  $ | M'_{\gpoint} | = | M_{\gpoint} |, b' = b+1$ and $ g' = g-1$.  }
 \end{figure}

 \smallskip
 
 \noindent \emph{Case 2:~$\gamma$ is a non-separating curve starting and ending on the same puncture.} 
 In this case,~$(S',M',P')$ is a marked surface with~$|M_{\gpoint}|$ marked points~$\gpoint$,~$|P|+1$ punctures,
 ~$b$ boundary components, and genus~$g-1$.  Induction applies.
 
 \smallskip
 
 \noindent \emph{Case 3:~$\gamma$ is a non-separating curve starting and ending on two different boundary components.} 
 In this case,~$(S',M',P')$ is a marked surface with~$|M_{\gpoint}|$ marked points~$\gpoint$,~$|P|$ punctures,
 ~$b-1$ boundary components, and genus~$g$.  Induction applies.
 
 \smallskip
 
 \noindent \emph{Case 4:~$\gamma$ is a non-separating curve starting and ending on two different punctures.} 
 In this case,~$(S',M',P')$ is a marked surface with~$|M_{\gpoint}|$ marked points~$\gpoint$,~$|P|-1$ punctures,
 ~$b$ boundary components, and genus~$g$.  Induction applies.

 \smallskip
 
 \noindent \emph{Case 5:~$\gamma$ is a non-separating curve starting on a boundary component and ending on a puncture.} 
 In this case,~$(S',M',P')$ is a marked surface with~$|M_{\gpoint}|$ marked points~$\gpoint$,~$|P|-1$ punctures,
 ~$b$ boundary components, and genus~$g$.  Induction applies.
 
 \smallskip
 
 \noindent \emph{Case 6:~$\gamma$ is a separating curve starting starting and ending on the same boundary component.} 
 In this case,~$(S',M',P')$ is a disjoint union of two marked surfaces, with a total of~$|M_{\gpoint}|$ marked points~$\gpoint$,~$|P|$ punctures,
 ~$b+1$ boundary components, and genus~$g$.  Induction applies (the $-2$ in the formula appears twice now). We again illustrate this case in an example, the other cases for separating curves being similar.
 \vspace{-.5cm}
 \begin{figure}[H]
\flushleft \includegraphics[scale=.45, angle=0]{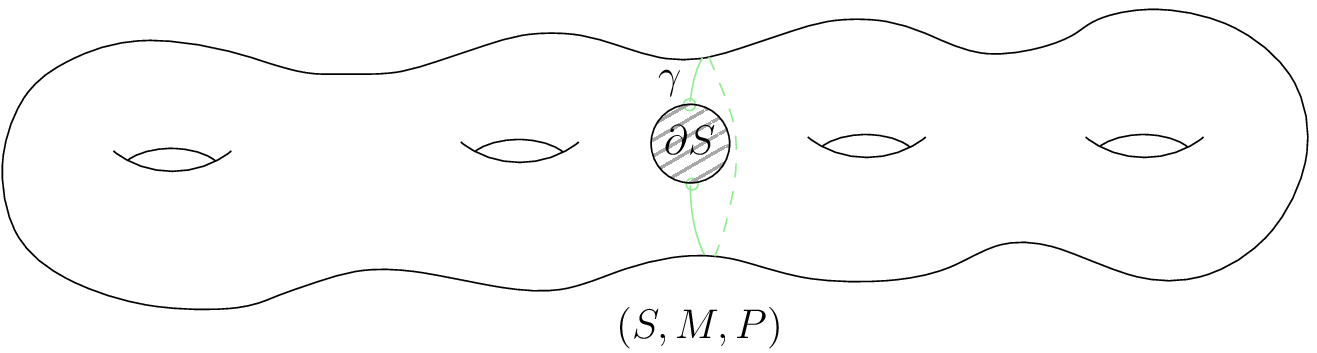}

\vspace{-.5cm} 
\centering \includegraphics[scale=.45,  angle=0]{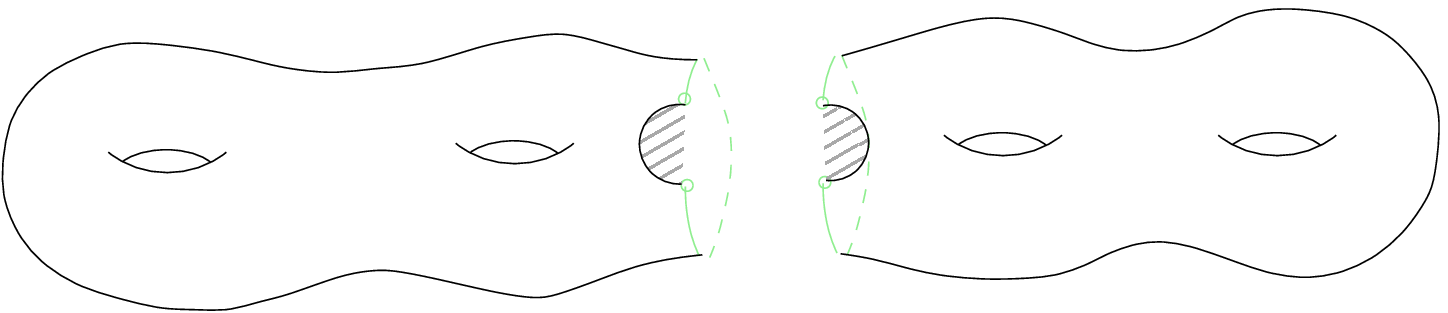}
\vspace{-.5cm}
\flushright \includegraphics[scale=.45, angle=0]{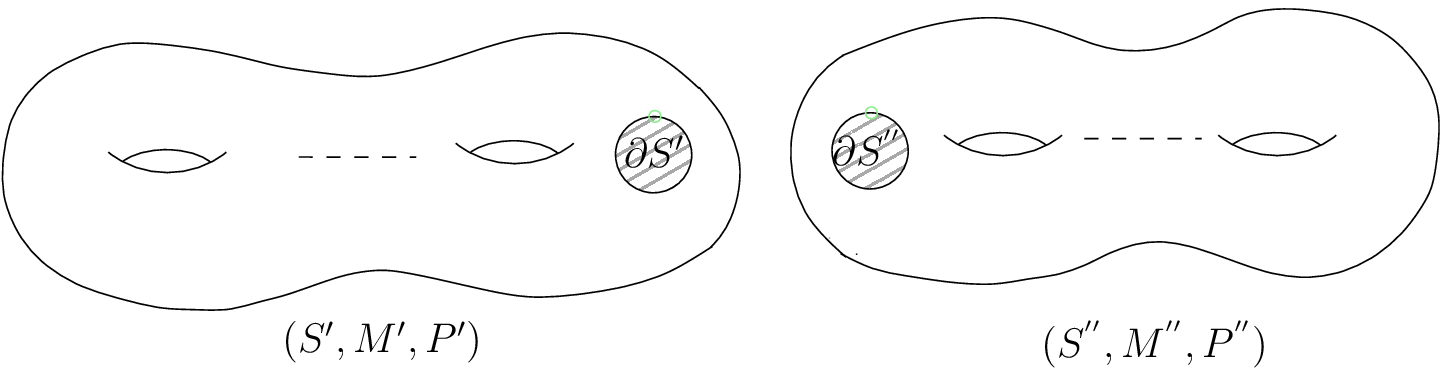}
\caption*{ Example of quotient of $(S, M, P)$ by a separating curve resulting in two disjoint surfaces $(S', M', P')$ and $(S'', M'', P'')$ where $| M'_{\gpoint} | + | M''_{\gpoint} |  = | M_{\gpoint} | -1$,  $b' +b'' = b+1$, $P'+P''=P$ and $ g' = k \leq g$ and  $ g'' = g-k$.  }
\end{figure}

 \smallskip
 
 \noindent \emph{Case 7:~$\gamma$ is a separating curve starting starting and ending on the same puncture.} 
 In this case,~$(S',M',P')$ is a disjoint union of two marked surfaces, with a total of~$|M_{\gpoint}|$ marked points~$\gpoint$,~$|P|+1$ punctures,
 ~$b$ boundary components, and genus~$g$.  Induction applies (the $-2$ in the formula appears twice now).
 
 All cases have been treated.
 \end{proof}

\begin{proposition}\label{prop::cellularDecomposition}
 Let~$\Delta$ be an admissible~$\gpoint$-dissection.  
 Then the complement of the~$\gpoint$-arcs of~$\Delta$ in~$S\setminus P$
 is a disjoint union of subsets homeomorphic to one of the following forms:
 \begin{enumerate}
  \item an open disc with precisely one~$\rpoint$-vertex in its boundary, or
  \item an open punctured disc with no~$\rpoint$-vertices on its boundary, and where the puncture corresponds
        to a~$\rpoint$-vertex in~$P_{\rpoint}$.
 \end{enumerate} 
 The above statement also holds if one permutes the~$\gpoint$ and~$\rpoint$ symbols.
\end{proposition}
\begin{proof}
 Let~$\mathcal P$ be a connected component of the complement of the~$\gpoint$-arcs of~$\Delta$ in $S\setminus P$.  By the admissibility condition,~$\mathcal P$ contains at least one~$\rpoint$-puncture or one~$\rpoint$-marked point on its boundary.
 
 If it has two or more, then it is possible to add a~$\gpoint$-arc separating two of them while still satisfying the admissibility condition.  This contradicts the maximality of~$\Delta$.  
 
 Thus~$\mathcal P$ contains exactly one~$\rpoint$-puncture or~$\rpoint$-marked point.  If the genus of~$\mathcal P$ were greater than~$0$, then a non-separating~$\gpoint$-arc could be added to~$\mathcal P$ without violating the admissibility condition, thus contradicting the maximality of~$\Delta$.  Therefore,~$\mathcal P$ has genus~$0$.
 
 Finally, if~$\mathcal P$ has one~$\rpoint$-puncture, then it has no boundary arcs, and thus is of type (2).  Otherwise,~$\mathcal P$ has one~$\rpoint$-marked point in its boundary and no puncture, and it is of type (1). 
\end{proof}

To any admissible~$\gpoint$-dissection, we can associate a dual~$\rpoint$-dissection in the following sense.
 
\begin{proposition}[{\cite[Prop. 1.16]{OpperPlamondonSchroll}} and {\cite[Prop. 3.6]{PaluPilaudPlamondon2}}]\label{prop::dualDissection}
 Let~$(S,M,P)$ be a marked surface, and let~$\Delta$ be an admissible~$\gpoint$-dissection.
 There exists a unique admissible~$\rpoint$-dissection~$\Delta^*$ (up to homotopy) such that 
 each arc of~$\Delta^*$ intersects exactly one arc of~$\Delta$.
\end{proposition}

\begin{definition}
 The dissection~$\Delta$ and~$\Delta^*$ are~\emph{dual dissections}.
\end{definition}

\begin{example}
Below is the dissection of Example~\ref{exam::disection} and its dual.
\begin{center}
 \begin{tikzpicture}
  \draw[thick] (0,0) circle (2);
  
  \draw[dark-green] (0,0) -- (0,2);
  \draw[dark-green] (0,0) -- (150:2);
  \draw[dark-green] (0,0) -- (-150:2);
  \draw[dark-green] (0,0) -- (-30:2);
  \draw[dark-green] (150:2) -- (-150:2);
  \draw[dark-green] (30:2) -- (-30:2);
  \draw[dark-green] (30:2) -- (90:2);
  \draw[dark-green] (-90:2) -- (-150:2);
  
  \draw[red] (1,0) -- (0:2);
  \draw[red] (1,0) -- (60:2);
  \draw[red] (1,0) -- (-60:2);
  \draw[red] (1,0) -- (120:2);
  \draw[red] (-1,0) -- (120:2);
  \draw[red] (-1,0) -- (180:2);
  \draw[red] (-1,0) -- (-60:2);
  \draw[red] (-120:2) -- (-60:2);
  
  \foreach \i in {0,...,5}
  {
  \filldraw[white] ({60*\i+30}:2) circle (0.1);
  \draw[thick, dark-green] ({60*\i+30}:2) circle (0.1);
  }
  \filldraw[white] (0,0) circle (0.1);
  \draw[thick, dark-green] (0,0) circle (0.1);
  
  \foreach \i in {0,...,5}
  {
  \filldraw[red] ({60*\i}:2) circle (0.1);
  }
  \filldraw[red] (-1,0) circle (0.1);
  \filldraw[red] (1,0) circle (0.1);
 \end{tikzpicture}
\end{center} 
\end{example} 
 
\subsection{The line field of an admissible dissection}


Let $\Delta$ be an admissible~$\gpoint$-dissection of a smooth marked surface $(S,M,P)$. The aim of this subsection is to associate a (homotopy class of a) line field $\eta(\Delta)$ to $\Delta$. 

In order to do so, we need the following basic lemma.

\begin{lemma}\label{lemma Delta eta}
Let $\cP\subset \mathbb R^2$ be a polygon with smooth oriented sides $\gamma_1,\ldots, \gamma_s$, and vertices $A_1=\gamma_1(0)=\gamma_s (1)$, $\ldots$, $A_s=\gamma_s(0)=\gamma_{s-1}(1)$ that are not necessarily pairwise distinct. Denote by $B_i=\gamma_i(\frac{1}{2})$ for $i=1,\ldots,s$, and let $C$ be a point in the interior of $\cP$.

For $i=1,\ldots, s$ denote by $\alpha_i$ a smooth simple curve in the interior of $\cP\setminus \{C\}$ from $B_i$ to $B_{i+1}$, normal to $\gamma_i$ and $\gamma_{i+1}$ in its endpoints, and so that $C$ is on the left. 

Then we have the following two statements.

\begin{enumerate}
\item There exists a line field $\theta_1$ defined on $\cP\setminus \{A_1,\ldots, A_s, \gamma_s\}$ such that 
\begin{enumerate}
\item $\theta_1$ is tangent to $\gamma_1$, $\ldots$, $\gamma_{s-1}$, and normal to $\gamma_s\setminus\{B_s\}$; 
\item for all $i=1,\ldots,s-1$, we have  $w_{\theta_1}(\alpha_i)=1$.
\end{enumerate}
Moreover such a line field is unique up to homotopy of line fields satisfying (a) and (b).

\item There exists a line field $\theta_2$ defined on $\cP\setminus\{A_1,\ldots,A_s,C\}$ such that
\begin{enumerate}
\item $\theta_2$ is tangent to $\gamma_1,\ldots, \gamma_s$;
\item for all $i=1,\ldots, s$ we have $w_{\theta_2}(\alpha_i)=1$.
\end{enumerate}
Moreover such a line field is unique up to homotopy of line fields satisfying (a) and (b).
\end{enumerate}
\end{lemma}

\begin{proof}
The foliations corresponding to $\theta_1$ and $\theta_2$ are drawn in the following pictures. The uniqueness is a direct consequence of Proposition \ref{propLF} (2).

\[\scalebox{1}{
\begin{tikzpicture}[>=stealth,scale=0.6]
\draw[thick, red] (0,0)--(-1,2)--(2,3)--(5,2)--(4,0);
\draw[thick](4,0)--(0,0);
\draw[red,thick,fill=red] (0,0) circle (0.15);
\draw[red,thick,fill=red] (-1,2) circle (0.15);
\draw[red,thick,fill=red] (2,3) circle (0.15);
\draw[red,thick,fill=red] (5,2) circle (0.15);
\draw[red,thick,fill=red] (4,0) circle (0.15);
\draw[dark-green,thick,fill=white] (2,0) circle (0.15);
\node at (-0.5,0) {$A_s$};
\node at (5.5,2) {$A_2$};
\node at (4.5,0) {$A_1$};
\node at (2,-0.5) {$B_s$};

\draw[thick, blue,<-] (-0.5,1)..controls (0,1.25) and (0.6,2)..(0.5,2.5);
\node[blue] at (0.5,1.2) {$\alpha_{s-2}$};

\draw[thick,red](0.5,0)..controls (0,1) and (-0.5,1.5)..(0,2).. controls (0.5,2.2) and (1,2.5)..(2,2.5)..controls (3,2.5) and   (3.5,2.2)..(4,2)..controls (4.5,1.5) and  (4,1)..(3.5,0);

\draw[thick,red] (1,0)..controls (0.5,0.8) and (1,2)..(2,2).. controls (3,2) and (3.5,0.8)..(3,0);

\draw[thick, red] (1.5,0)..controls (1.5,1) and (2.5,1)..(2.5,0);

\node at (2,-1.5) {the line field $\theta_1$};

\begin{scope}[xshift=9cm]
\draw[thick, red] (0,0)--(-1,2)--(2,3)--(5,2)--(4,0)--(0,0);

\draw[red,thick,fill=red] (0,0) circle (0.15);
\draw[red,thick,fill=red] (-1,2) circle (0.15);
\draw[red,thick,fill=red] (2,3) circle (0.15);
\draw[red,thick,fill=red] (5,2) circle (0.15);
\draw[red,thick,fill=red] (4,0) circle (0.15);

\node at (-0.5,0) {$A_s$};
\node at (5.5,2) {$A_2$};
\node at (4.5,0) {$A_1$};
\draw[dark-green,thick,fill=white] (2,1) circle (0.15);
\node[fill=white, inner sep=0pt] at (2.5,1) {$C$};
\draw[thick, red] (2,1) circle (0.8 and 0.4);
\draw[thick,red] (2,1.1) circle (1.2 and 0.6);
\draw[thick,red] (2, 1.2) circle (1.8 and 0.8);

\draw[thick, red](0,0.5)..controls (-0.5,1) and (-0.5,1.5).. (0,2).. controls (0.5,2.5) and (3.5,2.5) .. (4,2)..controls (4.5,1.5) and (4.5,1).. (4,0.5) .. controls (3.5,0) and (0.5,0).. (0,0.5);

\draw[thick, blue,<-] (-0.5,1)..controls (0,1.25) and (0.6,2) ..(0.5,2.5);
\node[blue, fill=white, inner sep=0pt] at (0.5,1.2) {$\alpha_{s-2}$};

\node at (2,-1.5) {the line field $\theta_2$};
\end{scope}

\end{tikzpicture}}\]

\end{proof}

Let $\Delta$ be an admissible dissection.
 The dissection $\Delta^*$ is also an admissible dissection, thus by Proposition~\ref{prop::cellularDecomposition}, the dissection $\Delta^*$ cuts  the surface $S \setminus P$ into a union of polygons with precisely one $\gpoint$ in their boundary  and once-punctured polygons with no $\gpoint$ on their boundary where the puncture corresponds to a $\gpoint$ in $P_{\gpoint}$. We define a line field $\eta(\Delta^*)$ on each of these polygons, following Lemma~\ref{lemma Delta eta}. It is defined as $\theta_1$ if the polygon has no puncture, and it is defined as $\theta_2$ if the polygon has one puncture. Since it is tangent to the sides of each polygon, it defines a line field on the surface $S\setminus P$. It is then unique up to homotopy. 

\begin{remark}
The line field defined above is different from the one defined in \cite{LekiliPolishchuk}. Indeed, the line field considered in \cite{LekiliPolishchuk} is tangent to the~$\gpoint$-dissection.
\end{remark}

We denote by $w^{\Delta^*}$ the winding number of the line field $\eta(\Delta^*)$.  It can be easily computed using the following rule. 

\begin{lemma}\label{lemm::computingWindingNumber}
 Let~$\eta$ be a line field satisfying the conditions in Lemma~\ref{lemma Delta eta}.
 Let~$\gamma$ be a smooth closed curve on~$S \setminus P$.
 Assume that~$\gamma$ intersects the arcs of~$\Delta^*$ orthogonally.
 Let~$t_0, t_1, \ldots, t_n=t_0$ be ordered on~$\bS^1$ so that for each~$i$, 
 $\gamma_i:=\gamma_{|_{(t_i,t_{i+1})}}$ is in one of the polygons or punctured polygons~$P_i$ bounded by the arcs of~$\Delta^*$ (see Proposition~\ref{prop::cellularDecomposition}).
 Assume that each~$\gamma_i$ is simple.
 Let
 \[
  w_i = \begin{cases}
         1 & \textrm{ if the $\gpoint$ is to the left of~$\gamma_i$ in~$P_i$,} \\
         -1 & \textrm{ if the $\gpoint$ is to the right of~$\gamma_i$ in~$P_i$.}
        \end{cases}
 \]
 Then~$w^{\Delta^*}(\gamma) = \sum_{i=0}^{n-1} w_i$.
\end{lemma}
\begin{proof} 
  The result follows from Proposition \ref{prop::compose winding}.
\end{proof}

\begin{remark}
 Any~$\gpoint$-arc is homotopic to a~$\gpoint$-arc which is orthogonal to the line field~$\eta$ of Lemma~\ref{lemma Delta eta} near its endpoints.
\end{remark}

For any pair of~$\gpoint$-arcs~$\gamma$ and~$\delta$ such that the ending point of~$\gamma$ is the starting point of~$\delta$,
    define their concatenation~$\gamma\delta$ as follows:
        let~$u,v\in(0,1)$ be such that~$\gamma$ is orthogonal to~$\eta$ on~$[u,1)$
        and~$\delta$ is orthogonal to~$\eta$ on~$(0,v]$.
        Let~$\gamma_0$ and~$\delta_0$ be the parts of~$\gamma$ and~$\delta$ defined on~$(0,u]$ and~$[v,1)$, respectively.
        Let~$\varepsilon$ be a simple curve that smoothly joins~$\gamma(u)$ to~$\delta(v)$. 
        Then the concatenation~$\gamma\delta$ is defined to be the concatenation of paths~$\gamma_0\varepsilon\delta_0$.

\begin{proposition}\label{prop::degreesOfArcs}
 Let~$(S,M,P)$ be a marked surface and let~$\Delta^*$ be an admissible~$\rpoint$-dissection.
 Let~$\gamma$ be a~$\gpoint$-arc or closed curve.
 \begin{enumerate}
  \item We have that~$w^{\Delta^*}(\gamma^{-1}) = -w^{\Delta^*}(\gamma)$.
  \item The integer~$w^{\Delta^*}(\gamma)$ only depends on the regular homotopy class of~$\gamma$ (where a \emph{regular homotopy} is a homotopy for which all intermediate curves are smooth).
  \item Assume that~$\gamma$ is a~$\gpoint$-arc, and let~$\delta$ be another~$\gpoint$-arc whose starting point is the ending point of~$\gamma$.      
        Then
        \[
         w^{\Delta^*}(\gamma\delta) = w^{\Delta^*}(\gamma) + w^{\Delta^*}(\delta) + \varepsilon,
        \]
        where
        \[
         \varepsilon = \begin{cases} 1 & \textrm{if the ending point of~$\gamma$ lies to the left of~$\gamma\delta$;} \\
                                     -1 & \textrm{if the ending point of~$\gamma$ lies to the right of~$\gamma\delta$.} \end{cases}
        \]

  \item Assume that~$\gamma$ is a~$\gpoint$-arc whose starting point and ending point are the same, and let~$\mathring{\gamma}$ be the corresponding closed curve.
        Then
        \[
         w^{\Delta^*}(\mathring{\gamma}) = w^{\Delta^*}(\gamma) + \varepsilon,
        \]
        where
        \[
         \varepsilon = \begin{cases} 1 & \textrm{if the ending point of~$\gamma$ lies to the left of~$\mathring{\gamma}$;} \\
                                     -1 & \textrm{if the ending point of~$\gamma$ lies to the right of~$\mathring{\gamma}$.} \end{cases}
        \]

  \item Assume that~$\gamma_1, \ldots, \gamma_b$ are simple closed curves that enclose a compact subsurface~$S'$ of~$S$, so that the~$\gamma_i$ are the boundary components of~$S'$ and~$S'$ has genus~$g'$.  Assume that the~$\gamma_i$ are oriented in such a way that~$S'$ lies to the right of each~$\gamma_i$.  Then
  \[
   \sum_{i=1}^b w^{\Delta^*}(\gamma_i) = 4-2b-4g'.
  \]

 \end{enumerate}
\end{proposition}
\begin{proof}
 Points~(1)-(4) directly follow from the definitions and Lemma~\ref{lemm::computingWindingNumber} adapted to arcs. Point~(5) is a re-statement in terms of line fields of \cite[Lemma 5.7]{Chillingworth}. 

\end{proof}

\section{Admissible dissections, gentle algebras and derived categories}\label{sect::admissible-dissections-gentle-algebras}

\subsection{The locally gentle algebra of an admissible dissection}\label{sect::locallyGentleAlgebras}

In this section, we recall some results of \cite{OpperPlamondonSchroll} that are needed for our main results.

\begin{definition}\label{defi::locallyGentleOfADissection}
 Let~$\Delta$ be an admissible~$\gpoint$-dissection of a marked surface~$(S,M,P)$.
 The $k$-algebra~$A(\Delta)$ is the quotient of the path algebra of the quiver $Q(\Delta)$ by the ideal $I(\Delta)$ defined as follows:
 \begin{itemize}
  \item the vertices of~$Q(\Delta)$ are in bijection with the~$\gpoint$-arcs in~$\Delta$.
  \item there is an arrow~$i\to j$ in~$Q(\Delta)$ whenever the~$\gpoint$-arcs~$i$ and~$j$ meet at a marked point~$\gpoint$,
  with~$i$ preceding~$j$ in the counter-clockwise order around~$\gpoint$, and with no other arc coming to~$\gpoint$ between~$i$ and~$j$;
  \item the ideal~$I(\Delta)$ is generated by the following relations: 
  whenever~$i$ and $j$ meet at a marked point as above, and the other end of~$j$ meets~$k$ at a marked point as above,
  then the composition of the corresponding arrows~$i\to j$ and~$j\to k$ is a relation.
 \end{itemize}
\end{definition}

\begin{example}
Below is the quiver with relations of the~$\gpoint$-dissection of Example~\ref{exam::disection}.
\begin{center}
 \begin{tikzpicture}
  \draw[thick] (0,0) circle (2);
  
  \draw[dark-green] (0,0) -- (0,2) node[blue,midway] (1) {1};
  \draw[dark-green] (0,0) -- (150:2) node[blue,midway] (5) {5};;
  \draw[dark-green] (0,0) -- (-150:2) node[blue,midway] (6) {6};;
  \draw[dark-green] (0,0) -- (-30:2) node[blue,midway] (4) {4};;
  \draw[dark-green] (150:2) -- (-150:2) node[blue,midway] (7) {7};;
  \draw[dark-green] (30:2) -- (-30:2) node[blue,midway] (3) {3};;
  \draw[dark-green] (30:2) -- (90:2) node[blue,midway] (2) {2};;
  \draw[dark-green] (-90:2) -- (-150:2) node[blue,midway] (8) {8};;
  
  \draw[blue, ->] (1) -- (5) ;
  \draw[blue, ->] (5) -- (6) node[midway] (e) {};
  \draw[blue, ->] (6) -- (7) node[midway] (f) {};
  \draw[blue, ->] (8) -- (6) node[midway] (h) {};
  \draw[blue, ->] (6) -- (4) node[midway] (i) {};
  \draw[blue, ->] (1) -- (2) node[midway] (a) {} ;
  \draw[blue, ->] (2) -- (3) node[midway] (b) {} ;
  \draw[blue, ->] (3) -- (4) node[midway] (c) {} ;
  \draw[blue, ->] (4) -- (1) node[midway] (d) {} ;
  \draw[blue, ->] (7) -- (5) node[midway] (g) {} ;
  
  \draw[blue, densely dotted] (a) -- (b) ;
  \draw[blue, densely dotted] (a) -- (d) ;
  \draw[blue, densely dotted] (b) -- (c) ;
  \draw[blue, densely dotted] (c) -- (d) ;
  \draw[blue, densely dotted] (e) -- (f) ;
  \draw[blue, densely dotted] (f) -- (g) ;
  \draw[blue, densely dotted] (g) -- (e) ;
  \draw[blue, densely dotted] (h) -- (i) ;
  
  \foreach \i in {0,...,5}
  {
  \filldraw[white] ({60*\i+30}:2) circle (0.1);
  \draw[thick, dark-green] ({60*\i+30}:2) circle (0.1);
  }
  \filldraw[white] (0,0) circle (0.1);
  \draw[thick, dark-green] (0,0) circle (0.1);
  
  \foreach \i in {0,...,5}
  {
  \filldraw[red] ({60*\i}:2) circle (0.1);
  }
  \filldraw[red] (-1,0) circle (0.1);
  \filldraw[red] (1,0) circle (0.1);
 \end{tikzpicture}
\end{center}
The dotted lines in the figure represent relations.
\end{example}

\begin{theorem}[\cite{OpperPlamondonSchroll}\cite{PaluPilaudPlamondon2}]\label{theo::bijectionDissectionsGentleAlgebras}
 The assignment~$\big((S,M,P), \Delta)\big) \mapsto A(\Delta)$ defines a bijection from
 the set of homeomorphism classes of marked surfaces~$(S,M,P)$ with an admissible dissection to
 the set of isomorphism classes of locally gentle algebras.
 Under this bijection, gentle algebras (that is the finite dimensional locally gentle algebras) correspond to the case where $P_{\gpoint}=\emptyset$.
\end{theorem}

\subsection{The surface as a model for the derived category}\label{sect::geometricModel}

We recall some of the results of~\cite{OpperPlamondonSchroll} on the correspondence of certain curves on $(S, M, P)$ and objects in the triangulated category $K^{-,b}(\proj A(\Delta))$ of complexes of finitely-generated projective~$A(\Delta)$-modules which are bounded on the right and whose total homology is bounded.
We will only recall those definitions which are needed in what follows and we refer to \cite{OpperPlamondonSchroll} for a complete description of $K^{-,b}(\proj A(\Delta))$ in terms of curves and intersections of curves in the associated surface.

\begin{definition} Let~$(S,M,P)$ be a marked surface with admissible~$\gpoint$-dissection~$\Delta$.
 Recall from Definition~\ref{defi::arcs} that a~$\gpoint$-arc is in particular a smooth map from~$(0,1)$ to~$S\setminus P$. In this definition, we assume that all arcs intersect the arcs of~$\Delta^*$ minimally and transversally.
A \emph{graded~$\gpoint$-arc}~$(\gamma,f)$ is a~$\gpoint$-arc~$\gamma$, together with a function
   \[
    f: \gamma \cap \Delta^* \longrightarrow \bZ,
   \]
  where~$\gamma \cap \Delta^*$ is the totally ordered set of intersection points of~$\gamma$ with~$\Delta^*$.
  The function~$f$ is required to satisfy the following:
  if~$p$ and~$q$ are in~$\gamma \cap \Delta^*$ and~$q$ is the successor of~$p$,
  then~$\gamma$ enters a polygon enclosed by~$\rpoint$-arcs of~$\Delta^*$ via~$p$ and leaves it via~$q$.
  If the~$\gpoint$ in this polygon is to the left of~$\gamma$, then~$f(q) = f(p)+1$; otherwise,~$f(q) = f(p)-1$.
\end{definition}

\begin{remark} 
 If~$(\gamma, f)$ is a finite~graded $\gpoint$-arc whose endpoints are not punctures, 
 and if~$p$ and $q$ are the first and last intersection points of~$\gamma$ with arcs of~$\Delta^*$, respectively,
 then by Lemma~\ref{lemm::computingWindingNumber}~$w^{\Delta^*}(\gamma) = f(q) - f(p)$, where~$w^{\Delta^*}$ is as in Definition~\ref{defwinding}.
\end{remark}

 In~\cite[Theorem 3.3]{OpperPlamondonSchroll}, a complete description of the indecomposable objects~in $K^{-,b}(\proj A(\Delta))$ was given in terms of graded curves; in particular, homotopy classes of graded~$\gpoint$-arcs are in bijection with certain indecomposable objects called \emph{(finite) string objects}.  We denote by~$\P_{(\gamma, f)}$ the object associated to the graded curve~$(\gamma,f)$.
 
 Furthermore, the morphisms between $\P_{(\gamma_1, f_1)}$ and $\P_{(\gamma_2, f_2)}$  correspond to the intersections of $\gamma_1$ and $\gamma_2$ such that the gradings $f_1$ and $f_2$ agree locally. More precisely, suppose that we have the following local configuration in $S$ depicted in the figure below,
 \begin{figure}[h!]
		\captionsetup{labelformat=empty}
		\captionsetup{justification=centering,singlelinecheck=false, format=hang}
		\centering
			{\begin{tikzpicture} 
				
				
				\foreach \u in {1,2} 
				\draw ({1.3*cos(360/8+360/4*\u)}, {1.3*sin(360/8+360/4*\u)})--({1.3*cos(180+360/8+360/4*\u)}, {1.3*sin(180+360/8+360/4*\u)}); 
				
				\foreach \u in {1,...,4} 
				\draw [line width=0.5, color=red] plot  [smooth, tension=1] coordinates {   ({2*cos(360/8+360/4*\u+360/16)}, {2*sin(360/8+360/4*\u+360/16)}) ({1.3*cos(360/8+360/4*\u)}, {1.3*sin(360/8+360/4*\u)}) ({2*cos(360/8+360/4*\u-360/16)}, {2*sin(360/8+360/4*\u-360/16)})};

				\draw ({1.75*cos(360/8)}, {1.75*sin(360/8)}) node {$\gamma_{2}$};
				\draw ({1.75*cos(-360/8)}, {1.75*sin(-360/8)}) node {$\gamma_{1}$};
				

				\draw ({1.3*cos(45)},{-1.3*sin(45)}) circle (2pt);
				\draw ({1.3*cos(45)},{1.3*sin(45)}) circle (2pt);
				\end{tikzpicture}
			} 
        \end{figure}        
where the four rounded curves are arcs of the dual dissection $\Delta^*$ which are not necessarily pairwise distinct, and where the intersection point of~$\gamma_1$ and~$\gamma_2$ may be on the boundary. The two empty circles designate points $v \in \gamma_1 \cap \Delta^*$ and $w \in \gamma_1 \cap \Delta^*$ such that  $f_1(v) = f_2(w)$.
Then there is a morphism from $\P_{(\gamma_1,f_1)}$ to $\P_{(\gamma_2,f_2)}$.

The~\emph{shift} of a graded curve~$(\gamma, f)$ is the graded curve~$(\gamma, f[1])$, where~$f[1] = f-1$.  It follows directly that if the intersection of $\gamma_1$ and $\gamma_2$ lies in the interior of $S$ then for a shift of the grading there is a morphism from $(\gamma_2,f_2)$ to $(\gamma_1,f_1[1]) $. 

\section{Silting objects}\label{sect::siltingObjects}
Let~$(S,M,P)$ be a marked surface with $P_{\gpoint}=\emptyset$, and let~$\Delta$ be an admissible~$\gpoint$-dissection.
Our aim in this section is to classify the silting objects in the bounded derived category of~$D^b(\MOD A(\Delta))$
in terms of graded curves on~$(S,M,P)$.

First, let us recall the definition of a silting object.

\begin{definition}
 Let~$A$ be a finite-dimensional~$k$-algebra.  An object~$X$ of~$D^b(\MOD A)$ is \emph{presilting} if
 \begin{itemize}
  \item $X$ is isomorphic to a bounded complex of finitely generated projective~$A$-modules;
  \item for any integer~$i>0$, $\Hom{D^b}(X, X[i]) =0$;
 \end{itemize}
 The object~$X$ is \emph{silting} if, moreover, $X$ generates the perfect derived category~$\perf A \cong K^b(\proj A)$.
 
 The object~$X$ is \emph{tilting} if it is silting and for any integer~$i<0$, $\Hom{D^b}(X, X[i]) =0$.
\end{definition}

Our main result in this section is the following.

\begin{theorem}\label{theo::siltingObjects}
 Let~$(S,M,P)$ be a marked surface, and let~$\Delta$ be an admissible~$\gpoint$-dissection.
 Let~$X$ be a basic silting object in~$D^b(\MOD A(\Delta))$.
 Then~$X$ is isomorphic to a direct sum~$\bigoplus_{i=1}^n \P_{(\gamma_i, f_i)}$,
 where $\{\gamma_1, \ldots, \gamma_n\}$ is an admissible~$\gpoint$-dissection of~$(S,M,P)$.
\end{theorem}

\begin{lemma}\label{lemm::numberOfSummands}
 The number of indecomposable direct summands of any silting object in~$D^b(\MOD A(\Delta))$ is~$|M_{\gpoint}|+|P|+b+2g-2$.
\end{lemma}
\begin{proof}
 All silting objects have the same number of indecomposable summands \cite[Corollary 2.28]{AiharaIyama}.
 The algebra~$A(\Delta)$ viewed as an object of~$D^b(\MOD A(\Delta))$ is a basic silting object,
 and its number of indecomposable direct summands is the number of~$\gpoint$-arcs in~$\Delta$.
 By Proposition~\ref{prop::numberOfArcs}, this number is~$|M_{\gpoint}|+|P|+b+2g-2$.
\end{proof}

\begin{lemma}\label{lemm::onlyFiniteArcsInSilting}
 Any indecomposable summand of a presilting object in~$D^b(\MOD A(\Delta))$ is isomorphic to an object of the form~$\P_{(\gamma,f)}$,
 where~$(\gamma, f)$ is a graded~$\gpoint$-arc.
\end{lemma}
\begin{proof}
 This is a consequence of~\cite[Theorem 2.12]{OpperPlamondonSchroll}.  More precisely, if~$\gamma$ is an infinite arc, then~$\P_{(\gamma,f)}$ is not in~$K^b(\proj A(\Delta))$.
 If~$\gamma$ is a closed curve, then any band object associated to it has self-extensions, 
 and thus cannot be a direct summand of a presilting object. 
\end{proof}

\begin{lemma}\label{lemm::noInteriorIntersection}
 Let~$(\gamma,f)$ and~$(\delta, g)$ be two graded~$\gpoint$-arcs, 
 and let~$\P_{(\gamma,f)}$ and~$\P_{(\delta, g)}$ be the corresponding objects.
 If~$\P_{(\gamma,f)}\oplus \P_{(\delta, g)}$ is presilting, then~$\gamma$ and~$\delta$ may only intersect at their endpoints.
\end{lemma}
\begin{proof}
 Assume that~$\gamma$ and~$\delta$ intersect in the interior of the surface.  Consider the local picture around such an intersection point.
  \begin{figure}[h!]
		\captionsetup{labelformat=empty}
		\captionsetup{justification=centering,singlelinecheck=false, format=hang}
		\centering
			{\begin{tikzpicture} 
				
				
				\foreach \u in {1,2} 
				\draw ({1.3*cos(360/8+360/4*\u)}, {1.3*sin(360/8+360/4*\u)})--({1.3*cos(180+360/8+360/4*\u)}, {1.3*sin(180+360/8+360/4*\u)}); 
				
				\foreach \u in {1,...,4} 
				\draw [line width=0.5, color=red] plot  [smooth, tension=1] coordinates {   ({2*cos(360/8+360/4*\u+360/16)}, {2*sin(360/8+360/4*\u+360/16)}) ({1.3*cos(360/8+360/4*\u)}, {1.3*sin(360/8+360/4*\u)}) ({2*cos(360/8+360/4*\u-360/16)}, {2*sin(360/8+360/4*\u-360/16)})};

				\draw ({1.75*cos(360/8)}, {1.75*sin(360/8)}) node {$\delta$};
				\draw ({1.75*cos(-360/8)}, {1.75*sin(-360/8)}) node {$\gamma$};
				

				\draw ({1.3*cos(45)},{-1.3*sin(45)}) circle (2pt);
				\draw ({1.3*cos(45)},{1.3*sin(45)}) circle (2pt);
				\draw ({-1.3*cos(45)},{1.3*sin(45)}) circle (2pt);
				\draw ({-1.3*cos(45)},{-1.3*sin(45)}) circle (2pt);
				
				\draw ({1.3*cos(45+15)},{-1.3*sin(45+15)}) node {$p$};
				\draw ({1.3*cos(45+18)},{1.3*sin(45+18)}) node {$q$};
				\draw ({-1.3*cos(45+18)},{1.3*sin(45+18)}) node {$p'$};
				\draw ({-1.3*cos(45+20)},{-1.3*sin(45+20)}) node {$q'$};
				\end{tikzpicture}
			} 
        \end{figure}
 
 If~$f(p)<g(q)$, then~$\Hom{D^b(\MOD A(\Delta))}(\P_{(\gamma, f)}, \P_{(\delta, g)}[g(q) - f(p)]) \neq 0$, contradicting the assumption that~$\P_{(\gamma,f)}\oplus \P_{(\delta, g)}$ is presilting.  Thus~$f(p) \geq g(q)$.  In a similar way, we prove that~$g(q)\geq f(p') \geq g(q') \geq f(p)$.  Therefore~$f(p) = f(p') = g(q) = g(q')$. But this contradicts the fact that~$f(p') = f(p) \pm 1 $ and~$g(q) = g(q') \pm 1$.

\end{proof}

\begin{lemma}\label{lemm::noInternalPolygons}
 Let~$(\gamma_1, f_1), \ldots, (\gamma_r, f_r)$ be pairwise distinct graded~$\gpoint$-arcs such that the object
 ~$\P_{(\gamma_1,f_1)}\oplus \ldots \oplus \P_{(\gamma_r, f_r)}$ is basic presilting.
 Then~$\{\gamma_1, \ldots, \gamma_r\}$ is an admissible collection of~$\gpoint$-arcs.
\end{lemma}
\begin{proof}
 By Lemma~\ref{lemm::noInteriorIntersection}, the arcs~$\gamma_1, \ldots, \gamma_r$ are pairwise non-intersecting, except possibly at their endpoints.
 It remains to be shown that they do not enclose any unpunctured surface.
 Assume that arcs~$\gamma_1, \ldots, \gamma_s$ do enclose an unpunctured surface~$S'$.  
 We can orient the arcs so that the surface~$S'$ lies to the right of each~$\gamma_i$.
 Say that~$S'$ has~$b$ boundary components.  Let us re-index the~$\gamma_i$ in such a way that the~$j$-th boundary component consists of the arcs~$\gamma_{j,1}, \ldots, \gamma_{j,s_j}$, and let~$\delta_j$ be the concatenation of the~$\gamma_{j,i}$.
 
 Then by Proposition~\ref{prop::degreesOfArcs}(5), we have that~$$\sum_{j=1}^b w^{\Delta^*}(\delta_j) = 4-2b-4g'.$$
 
 Now, let~$p_{j,i}$ and~$q_{j,i}$ be the first and last intersection points of~$\gamma_{j,i}$ with arcs of~$\Delta^*$, for all~$j\in\{1, \ldots, b\}$ and all~$i\in \{1, \ldots, s_j\}$.
 Using Proposition~\ref{prop::degreesOfArcs}(3) and (4), we have that
 \begin{eqnarray*}
  \sum_{j=1}^b w^{\Delta^*}(\delta_j) & = & \sum_{j=1}^b \big( s_j + \sum_{i=1}^{s_j} w^{\Delta^*}(\gamma_{j,i}) \big) \\
                         & = & \sum_{j=1}^b \Big( s_j + \sum_{i=1}^{s_j} \big( f_{j,i}(q_{j,i}) - f_{j,i}(p_{j,i}) \big)\Big).
 \end{eqnarray*}
 
  Finally, by the definition of a silting object and the description of morphisms between the objects~$\P_{(\gamma_i, f_i)}$ (see \cite[Theorem 3.3]{OpperPlamondonSchroll}),
 we have that for all~$j\in\{1, \ldots, b\}$ and all~$i\in \{1, \ldots, s_j\}$ (taken modulo~$s_j$), $f_{j,i+1}(p_{j,i+1}) \leq f_{j,i}(q_{j,i})$, otherwise there would be a non-zero morphism in
 $\Hom{D^b}(\P_{(\gamma_{j,i}, f_{j,i})}, \P_{(\gamma_{j,i+1}, f_{j,i+1})}[\ell])$ for some~$\ell>0$.
  Hence
 \[
  4-2b-4g' = \sum_{j=1}^b w^{\Delta^*}(\delta_j) = \sum_{j=1}^b \Big( s_j + \sum_{i=1}^{s_j} \big( f_{j,i}(q_{j,i}) - f_{j,i}(p_{j,i}) \big)\Big) \geq \sum_{j=1}^b s_j \geq b.
 \]
 Since the right-hand side is positive, we must have that~$g'=0$ and~$b=1$, so~$S'$ is a disc.  Moreover, we deduce that~$s\leq 2$.  But~$s=1$ is impossible, otherwise~$S'$ would be a monogon, so the only curve on its boundary would be contractible in~$S$, a contradiction.  Thus~$s=2$.  But then~$\gamma_1$ and~$\gamma_2$ enclose an unpunctured digon, so~$\gamma_2 = \gamma_1^{-1}$.
 Thus~$\P_{(\gamma_1,f_1)}$ and~$\P_{(\gamma_2,f_2)}$ are isomorphic up to shift, 
 and the only way for~$\P_{(\gamma_1,f_1)}\oplus \P_{(\gamma_2,f_2)}$ to be presilting is for the two objects to be isomorphic.
 This contradicts the fact that the graded~$\gpoint$-arcs are pairwise distinct.
  
 Thus the~$\gpoint$-arcs~$\gamma_1, \ldots, \gamma_r$ do not enclose any unpunctured surface, and they form an admissible~$\gpoint$-dissection.
 
\end{proof}

\begin{proof}[Proof of Theorem~\ref{theo::siltingObjects}]
 Let~$X$ be a basic silting object.
 By Lemma~\ref{lemm::numberOfSummands}, it has~$n=|M_{\gpoint}|+|P|+b+2g-2$ indecomposable direct summands.
 By Lemma~\ref{lemm::onlyFiniteArcsInSilting}, there are graded~$\gpoint$-arcs~$(\gamma_1, f_1), \ldots, (\gamma_n, f_n)$
 such that~$X$ is the direct sum of all the~$\P_{(\gamma_i, f_i)}$.
 By Lemma~\ref{lemm::noInteriorIntersection},~$\{\gamma_1, \ldots, \gamma_n\}$ is an admissible collection of~$\gpoint$-arcs.
 By Proposition~\ref{prop::numberOfArcs}, this collection is an admissible~$\gpoint$-dissection.
\end{proof}

A converse of Theorem~\ref{theo::siltingObjects} can be stated by using the following result.

\begin{proposition}\label{prop::maximal-presilting-are-silting}
 Let~$A$ be a gentle algebra, and let~$X$ be a presilting object in~$D^b(\MOD A)$.  Then~$X$ is silting if and only if it is maximal presilting (in the sense that if~$X\oplus X'$ is presilting, then~$X'\in \add(X)$).
\end{proposition}
\begin{proof}
 Over any algebra, a silting object is always maximal presilting.  We need to prove the converse for gentle algebras.  Let~$(S,M,P)$ be the marked surface associated to~$A$.  Assume that~$X$ is maximal presilting.  We need to show that~$X$ generates~$K^b(\proj(A))$.  Without loss of generality, we can assume that~$X$ is basic.  Using Lemma~\ref{lemm::onlyFiniteArcsInSilting}, we can write
 \[
  X = \bigoplus_{i=1}^m \P_{(\gamma_i, f_i)}.
 \]
 Combining Lemmas~\ref{lemm::noInteriorIntersection} and~\ref{lemm::noInternalPolygons}, we get that~$\gamma_1, \ldots, \gamma_m$ form an admissible~$\gpoint$-dissection.  By Proposition~\ref{prop::numberOfArcs}, we get that~$m$ is the number of pairwise non-isomorphic indecomposable direct summands of~$A$.
 
 Without loss of generality, we can assume that~$A$ is basic.  Write~$A = \bigoplus_{i=1}^m \P_{(\delta_i, g_i)}$.  It suffices to show that a shift of each~$\P_{(\delta_i, g_i)}$ is in the triangulated category generated by~$X$.
 
 This is achieved as follows.  Any~$\delta_i$ is the concatenation of the~$\gamma_j$.  To see this, note that the~$\gamma_j$ cut the surface into discs by Proposition~\ref{prop::cellularDecomposition}.  An arc crossing a disc is homotopic to a concatenation of some of the segments forming the boundary of this disc; applying this to all discs crossed by~$\delta_i$, we get that~$\delta_i$ is a concatenation of the~$\gamma_j$.  Finally, by~\cite[Theorem 4.1]{OpperPlamondonSchroll}, concatenation of arcs corresponds to taking the cones of morphisms between the associated objects (up to a shift) in~$D^b(\MOD A)$.  Thus the~$\P_{(\delta_i, g_i)}$ are in the triangulated category generated by the~$\P_{(\gamma_i, f_i)}$.
\end{proof}

\begin{corollary}\label{coro::tilting-silting}
 Let~$A$ be a gentle algebra with associated marked surface~$(S,M,P)$ and~$\gpoint$-dissection~$\Delta$.  Let~$(\gamma_1, f_1) \ldots, (\gamma_r, f_r)$ be graded $\gpoint$-arcs such that~$\gamma_1, \ldots, \gamma_r$ form an admissible~$\gpoint$-dissection~$\Delta'$ of~$(S,M,P)$.  
 
 For any~$\gpoint$-marked point, let~$\gamma_{i_1}, \ldots, \gamma_{i_s}$ be the arcs of~$\Delta'$ ending in that marked point in counter-clockwise order, and let~$p_{i_1}, \ldots, p_{i_s}$ be their respective intersection with the dual~$\rpoint$-dissection~$\Delta^*$ closest to the~$\gpoint$-marked point.
 \begin{enumerate}
  \item If, for every~$\gpoint$-marked point, we have that
  \[
   f_{i_1}(p_{i_1}) = \ldots = f_{i_s}(p_{i_s}),
  \]
  then~$\bigoplus_{i=1}^r \P_{(\gamma_i, f_i)}$ is a tilting object.
 
  \item If, for every~$\gpoint$-marked point, we have that
  \[
   f_{i_1}(p_{i_1}) \geq \ldots \geq f_{i_s}(p_{i_s}),
  \]
  then~$\bigoplus_{i=1}^r \P_{(\gamma_i, f_i)}$ is a silting object.

 \end{enumerate}

\end{corollary}

\section{Derived invariants for gentle algebras}\label{sect::derivedInvariants}

In this Section, combining the results of the previous sections, we show when two gentle algebras are derived equivalent. Namely, we prove the following. 

\begin{theorem}\label{theo::derivedInvariants}
 Let~$A$ and~$A'$ be gentle algebras, and let~$(S,M,P)$ and~$(S',M',P')$ be marked surfaces with no~$\gpoint$ punctures and with admissible~$\rpoint$-dissections~$\Delta^*$ and~$\Delta'^*$ associated to~$A$ and~$A'$, respectively.  
 Then~$A$ and~$A'$ are derived equivalent if and only if
 there exists an orientation-preserving homeomorphism~$\Phi:(S,M,P)\to (S',M',P')$ such that for any simple closed curve~$\delta$ on~$(S,M,P)$, we have
 \[
  w^{\Delta'^*}\big(\Phi(\delta)\big) = w^{\Delta^*}(\delta).
 \]

\end{theorem}
\begin{proof}
 Assume that~$A$ and~$A'$ are derived equivalent.
 Then by Rickard's theorem \cite{Rickard}, there exists a tilting object~$T$ in~$D^b(A)$ whose endomorphism ring is isomorphic to~$A'$.
 By Theorem~\ref{theo::siltingObjects}, there are graded curves \sloppy $(\gamma_1, f_1), \ldots, (\gamma_n, f_n)$ on~$(S,M,P)$ such that
 $T \cong \bigoplus_{i=1}^n \P_{(\gamma_i, f_i)}$. 
 By Theorem~\ref{theo::siltingObjects}, the curves~$\gamma_1 \ldots, \gamma_n$ form an admissible~$\gpoint$-dissection~$\Delta_T$ of~$(S,M,P)$.
 Moreover, using the description of the morphisms in the derived category given in~\cite[Theorem 3.3]{OpperPlamondonSchroll}, one obtains that the algebra~$A(\Delta_T)$ (see Definition~\ref{defi::locallyGentleOfADissection}) is isomorphic to~$A'$.
 By Theorem~\ref{theo::bijectionDissectionsGentleAlgebras}, there exists an orientation-preserving homeomorphism~$\Phi:(S,M,P)\to (S',M',P')$ sending~$\Delta_T$ to~$\Delta'$.
 
 Let~$\delta$ be any closed curve on~$(S,M,P)$.
 Then~$\delta$ is isotopic to a concatenation of arcs~$\delta_1, \ldots, \delta_s$ from the dissection~$\Delta$. 
 By Proposition~\ref{prop::degreesOfArcs}, we have
 \begin{eqnarray*}
  w^{\Delta^*}(\delta) &=& \ell - r + \sum_{i=1}^s  w^{\Delta^*}(\delta_i),
 \end{eqnarray*}
 where~$\ell$ and~$r$ are the number of~$i\in \{1, \ldots, s\}$ (taken modulo~$s$) such that the endpoint of~$\delta_i$ is to the left or to the right, respectively, of the concatenation~$\delta_i\delta_{i+1}$.
 Note that for each~$i$, we have that~$w^{\Delta^*}(\delta_i) = 0$, since the~$\delta_i$ are part of the initial~$\gpoint$-dissection~$\Delta$.
 Thus
 \begin{eqnarray*}
  w^{\Delta^*}(\delta) &=& \ell - r.
 \end{eqnarray*}

 Similarly,~$\delta$ is isotopic to a concatenation of arcs~$\gamma_{j_1}, \ldots, \gamma_{j_u}$ from the dissection~$\Delta_T$. 
 Let~$p_i$ and~$q_i$ be the first and last intersection point of~$\gamma_{j_i}$ with arcs of~$\Delta_T^*$.
 By Proposition~\ref{prop::degreesOfArcs}, we get that
 \begin{eqnarray*}
  w^{\Delta^*}(\delta) &=& \ell_T - r_T + \sum_{i=1}^u  w^{\Delta^*}( \gamma_{j_i}) \\
                         &=& \ell_T - r_T + \sum_{i=1}^u \big( f_{j_i}(q_i) - f_{j_{i}} (p_i) \big), 
 \end{eqnarray*}
 where~$\ell_T$ and~$r_T$ are the number of~$i\in \{1, \ldots, u\}$ (taken modulo~$u$) such that the endpoint of~$\gamma_{j_i}$ is to the left or to the right, respectively, of the concatenation~$\gamma_{j_i}\gamma_{j_{i+1}}$.
 
 Since~$T$ is tilting, we have that~$f_{j_i}(q_i) = f_{j_{i+1}}(p_{i+1})$ for all~$i$. Thus~$\sum_{i=1}^u \big( f_{j_i}(q_i) - f_{j_{i}} (p_i) \big) = 0$, so
 \begin{eqnarray*}
  w^{\Delta^*}(\delta) &=& \ell_T - r_T.
 \end{eqnarray*}
 
 Finally, consider the simple closed curve~$\Phi(\delta)$ on~$(S',M',P')$.
 It is a concatenation of the~$\gpoint$-arcs~$\Phi(\gamma_{j_1}), \ldots, \Phi(\gamma_{j_u})$.
 Since these arcs are in the initial~$\gpoint$-dissection~$\Delta'$ of~$(S',M',P')$, a similar calculation yields
 \begin{eqnarray*}
  w^{\Delta^*}(\delta) &=& \ell' - r' + \sum_{i=1}^u  w^{\Delta'^*}(\Phi(\gamma_{j_i})) \\
                         &=& \ell' - r',
 \end{eqnarray*}
  where~$\ell'$ and~$r'$ are the number of~$i\in \{1, \ldots, u\}$ (taken modulo~$u$) such that the endpoint of~$\Phi(\gamma_{j_i})$ is to the left or to the right, respectively, of the concatenation~$\Phi(\gamma_{j_i})\Phi(\gamma_{j_{i+1}})$.
  Since~$\Phi$ is a homeomorphism, we have that~$\ell' = \ell_T$ and~$r' = r_T$.
  Thus
  \[
   w^{\Delta'^*}(\Phi(\delta)) = \ell' - r' = \ell_T - r_T = w^{\Delta^*}(\delta).
  \]
  
  \medskip
  
  Assume now that there exists a~$\Phi$ as in the statement of the theorem.
  Denote by~$\tau'_1, \ldots, \tau'_n$ the~$\gpoint$-arcs of~$\Delta'$, by~$\tau_1, \ldots, \tau_n$ their preimages by~$\Phi$, and let~$g'_1, \ldots, g'_n$ be gradings such that~$A' = \bigoplus_{k=1}^n \P_{(\tau'_k, g'_k)}$.
  
  It is clear that~$A(\Phi^{-1}(\Delta')) \cong A'$.  If we can show that there exist gradings~$g_1, \ldots, g_n$ on the~$\gpoint$-arcs~$\tau_1, \ldots, \tau_n$ such that~$T = \bigoplus_{k=1}^n \P_{(\tau_k, g_k)}$ is a tilting object, then we would have that~$\End{D^b(A)}(T) \cong A'$ would be derived equivalent to~$A$, and the theorem would be proved.
  
  To construct such gradings, we first recursively associate an integer~$n(x)$ to each~$\gpoint$ marked point~$x$ of~$(S,M,P)$ as follows.  Let~$x_0$ be any~$\gpoint$ marked point, and let~$n(x_0) := 0$.  For any~$\gpoint$ marked point~$x$, there exists a path~$\tau_{i_1}^{\varepsilon_1}\cdots \tau_{i_r}^{\varepsilon_r}$ from~$x_0$ to~$x$  (where each~$\varepsilon_{i}$ is a~$\pm 1$), since~$\Delta'$ is connected.  Let~
  \[
   n(x) := \sum_{k=1}^r w^{\Delta^*}(\tau_{i_k}^{\varepsilon_k}).
  \]
  This integer does not depend on the choice of a path from~$x_0$ to~$x$: if~$\tau_{j_1}^{\eta_1}\cdots\tau_{j_s}^{\eta_s}$ is another such path, then let~$\gamma$ be the concatenation~$\tau_{j_1}^{\eta_1}\cdots\tau_{j_s}^{\eta_s}\tau_{i_r}^{-\varepsilon_r}\cdots \tau_{i_1}^{-\varepsilon_1}$.  Then by the hypothesis on~$\Phi$, we get that
  \[
   w^{\Delta^*}(\mathring{\gamma}) = w^{\Delta'^*}(\Phi(\mathring{\gamma})),
  \]
  where~$\mathring{\gamma}$ is as in Proposition~\ref{prop::degreesOfArcs}(4).
  Applying Lemma~\ref{lemm::computingWindingNumber} on both sides of this equation, we get that
  \[
   \sum_{k=1}^s w^{\Delta^*}(\tau_{j_k}^{\eta_k}) - \sum_{\ell=1}^{r} w^{\Delta^*}(\tau_{i_\ell}^{\varepsilon_\ell}) = \sum_{k=1}^s w^{\Delta'^*}(\tau_{j_k}'^{\eta_k}) - \sum_{\ell=1}^{r} w^{\Delta'^*}(\tau_{i_\ell}'^{\varepsilon_\ell}).
  \]
  Notice that the right-hand side of this equation is zero: indeed, since the arcs~$\tau'_i$ are part of the dissection~$\Delta'$, we have that~$w^{\Delta'^*}(\tau'_i) = 0$ for all~$i\in\{1, \ldots, n\}$.  Therefore
  \[
   \sum_{k=1}^s w^{\Delta^*}(\tau_{j_k}^{\eta_k}) = \sum_{\ell=1}^{r} w^{\Delta^*}(\tau_{i_\ell}^{\varepsilon_\ell}),
  \]
  so the integer~$n(x)$ does not depend on the choice of path from~$x_0$ to~$x$.
  
  Using the integers~$n(x)$, define the gradings~$g_1, \ldots, g_n$ on~$\tau_1, \ldots, \tau_n$ in such a way that for all~$i\in \{1, \ldots, n\}$, if~$x_i$ is the starting point of~$\tau_i$ and~$p_i$ is the intersection point of~$\tau_i$ with~$\Delta^*$ closest to~$x_i$, then~$g_i(p_i) = n(x_i)$.
  
  Let~$y_i$ be the endpoint of~$\tau_i$, and let~$q_i$ be the intersection point of~$\tau_i$ with~$\Delta^*$ closest to~$y_i$.  If we prove that~$g_i(q_i) = n(y_i)$, then by Corollary~\ref{coro::tilting-silting}, we would have that~$\bigoplus_{i=1}^n \P_{(\tau_i, g_i)}$ is a tilting object.  But
  \begin{eqnarray*}
   g_i(q_i) & = & g_i(p_i) + w^{\Delta^*}(\tau_i) \\
            & = & n(x_i) + w^{\Delta^*}(\tau_i) \\
            & = & \sum_{j=1}^r w^{\Delta^*}(\tau_{i_j}^{\varepsilon_j}) + w^{\Delta^*}(\tau_i) \quad \textrm{(for any path $\prod_{j=1}^r \tau_{i_j}^{\varepsilon_j}$ from~$x_0$ to~$x_i$)} \\
            & = & n(y_i).
  \end{eqnarray*}
  Therefore~$\bigoplus_{i=1}^n \P_{(\tau_i, g_i)}$ is a tilting object.  This proves that~$A$ and~$A'$ are derived equivalent.
  
\end{proof}

\begin{remark}
 Restricting Theorem~\ref{theo::derivedInvariants} to the closed curves circling boundary components of~$(S,M,P)$, 
 we reobtain the derived invariant of D.~Avella-Alaminos and~C.~Geiss \cite{AvellaAlaminosGeiss} by using \cite{OpperPlamondonSchroll}.
 
 Indeed, if $c$ is a curve surrounding a boundary component or a puncture (and having the boundary on the left), then $\Phi(c)$ is also a curve surrounding a boundary component, and the number of marked points on the respective boundary components coincide.  Then it is clear that the collections of pairs $(n_j,n_j-w^{\Delta^*}(c_j))$ for $j=1,\ldots b+p$ where $n_j$ is the number of~$\gpoint$ marked points on the boundary component attached to $c_j$ is a derived invariant. It is the AG invariant by \cite{OpperPlamondonSchroll}. 
\end{remark}

\section{Numerical derived invariants via Arf invariants}\label{sect::MCG}

\subsection{Action of the mapping class group}

The following section recalls results of \cite{LekiliPolishchuk}, that are a generalisation  from vector fields to line fields of results due to Kawazumi \cite{Kawazumi} that mainly follow from \cite{Johnson, Arf}).

Define the mapping class group of $\Sigma$ as 
$$\MCG (\Sigma)=\Diff^{+,\partial\Sigma}(\Sigma)/\Diff^{+,\partial\Sigma}_0(\Sigma)$$ where $\Diff^{+,\partial\Sigma}(\Sigma)$ is the group of orientation-preserving diffeomorphisms on $\Sigma$ that are the identity pointwise on the boundary, and $\Diff_0^{+,\partial\Sigma}(\Sigma)$ is the subgroup  of those isotopic to the identity. 

The mapping class group acts on the set $\LF(\Sigma)$ (see Proposition~\ref{propLF}). Indeed, to a line field $\eta$ and to a diffeomorphism $\Phi$, one can define the pullback of $\eta$ by $\Phi$  as  $$\Phi^*(\eta)(x)=[(T_x\Phi)^{-1}](\eta\circ \Phi(x)) \quad \forall x\in \Sigma.$$
If $\Phi$ is isotopic to the identity, then $\eta$ and $\Phi^*(\eta)$ are homotopic line fields, hence the action is well-defined.

\begin{lemma}\label{lemma w}
Let $\eta$ be a line field on $\Sigma$, and $\Phi\in \Diff^{+,\partial\Sigma}(\Sigma)$. For any $f\in\cC^1(\mathbb{S}^1,\Sigma)$ we have $w_{\Phi^*(\eta)}(f)=w_{\eta}(\Phi_*(f))$.
\end{lemma}

\begin{proof} An immediate computation gives for each $z\in \bS^1$ 
$$Z^{\Phi_*(f)}(z)=[T_{f(z)}\Phi\circ T_zf(1)],\quad X^{\Phi_*(f),\eta}(z)=\eta\circ\Phi\circ f(z),$$
$$Z^f(z)=[T_zf(1)]\quad\textrm{and}\quad X^{f,\Phi^*(\eta)}(z)=[(T_{f(z)}\Phi)^{-1}](\eta\circ\Phi\circ f(z)).$$
Hence we have $X^{\Phi_*(f),\eta}=[T \Phi]_*(X^{f,\Phi^*(\eta)}$ and $Z^{\Phi_*(f)}=[T\Phi]_*(Z^f).$ Since the following diagram is commutative we get the result.

\[ \xymatrix{ 1 \ar[r] &\pi_1(\mathbb{S}^1,1)\ar[r]^-{\iota_*} \ar@{=}[d]
& \pi_1(\mathbb{P}(T\Sigma),[v_0])\ar[r]^-{p_*}\ar[d]^{[T\Phi]_*} & \pi_1(\Sigma,x_0)\ar[d]^{\Phi_*}\ar[r] & 1\\
1\ar[r] &\pi_1(\mathbb{S}^1,1)\ar[r]^-{\iota_*} 
& \pi_1(\mathbb{P}(T\Sigma),[T_{x_0}\Phi (v_0)])\ar[r]^-{p_*} & \pi_1(\Sigma,\Phi(x_0))\ar[r] & 1
 }\]

\end{proof}

\begin{remark} Combining the above result together with Proposition \ref{propLF} (2), Theorem \ref{theo::derivedInvariants}  can be reformulated as Theorem \ref{theo::intro}, that is, two gentle algebras $A(\Delta)$ and $A(\Delta')$ are derived equivalent if and only if there exists a homeomorphism of marked surfaces $\Phi:(S,M,P)\to (S',M',P')$ such that the line fields $\Phi^*(\eta(\Delta^*))$ and $\eta(\Delta'^*)$ are homotopic.
\end{remark}

Let $\Sigma$ be a smooth surface of genus $g$, with $b$ boundary components and $p$ punctures.
We denote by $\partial_1 \Sigma,\ldots,\partial_b\Sigma$ the boundary components of $\Sigma$.

Denote by $\cB=\{c_1,\ldots,c_{b+p}\}$ a set of simple closed curves such that for $j=1,\ldots b$, $c_j$ is homotopic to the boundary component $\partial_j \Sigma$ (being on the left of the curve), and so that $c_{b+k}$ is homotopic to a circle around the $k$-th puncture for $k=1,\ldots, p$. Let denote $\overline{\Sigma}$ the closed surface with empty boundary obtained by adding closed discs to each bounday component. Let 
$\cG=\{\alpha_1,\beta_1,\ldots,\alpha_g,\beta_g\}$be a set of closed simple curves, such that their image in $H_1(\overline{\Sigma},\mathbb Z)$ is a symplectic basis (with respect to the intersection form).

\[\scalebox{0.8}{\begin{tikzpicture}[scale=0.6, >=stealth]
\shadedraw[top color= blue!30] (-3,0)..controls (-3,1) and (-1,2)..(0,2)..controls (1,2) and (3,1).. (4,1)..controls (5,1) and (7,2).. (8,2)..controls (9,2) and (11,1)..(12,1)..controls (13,1) and (14,2).. (14,3)..controls (14,2.5) and (16,2.5)..(16,3)..controls (16,1) and (18,1)..(18,3)..controls (18,2.5) and (20,2.5).. (20,3)..controls (20,0) and  (18.1,-1)..(18.1,-3).. controls (18.1, -2.9) and (17.9,-2.9).. (17.9,-3).. controls (17.9,0) and (15.1,0).. (15.1,-3).. controls (15.1,-2.9) and (14.9,-2.9).. (14.9,-3).. controls (14.9,-1) and 
 (14,-1)..(12,-1).. controls (11,-1) and (9,-2)..(8,-2).. controls (7,-2) and (5,-1).. (4,-1).. controls (3,-1) and (1,-2).. (0,-2)..controls (-1,-2) and (-3,-1).. (-3,0);

\draw (14,3)..controls (14,3.5) and (16,3.5)..(16,3);
\draw (18,3)..controls (18,3.5) and (20,3.5).. (20,3);

\draw[fill=white] (-1,0)..controls (0,-0.3) and (0,-0.3)..(1,0)..controls (0,0.3) and (0,0.3).. (-1,0);
\draw (-1,0)--(-1.2,0.06);
\draw (1,0)--(1.2,0.06);

\draw[fill=white] (7,0)..controls (8,-0.3) and (8,-0.3)..(9,0)..controls (8,0.3) and (8,0.3).. (7,0);
\draw (7,0)--(6.8,0.06);
\draw (9,0)--(9.2,0.06);

\draw[thick] (0,0) circle (2 and 1);
\draw[thick,->] (2,0.1) --(2,-0.1);
\node at (2.5, 0) {$\alpha_1$};

\draw[thick] (0,-0.2)..controls (0.5,-0.2) and (0.5,-2)..(0,-2);
\draw[thick,<-](0.38,-1.2)--(0.38,-1.4);
\node at (0.8,-1.4) {$\beta_1$};

\draw[thick] (8,0) circle (2 and 1);
\draw[thick,->] (10,0.1) --(10,-0.1);
\node at (10.5, 0) {$\alpha_2$};
\draw[thick] (8,-0.2)..controls (8.5,-0.2) and (8.5,-2)..(8,-2);
\draw[thick,<-](8.38,-1.2)--(8.38,-1.4);
\node at (8.8,-1.4) {$\beta_2$};


\draw[thick] (13.9,2.5)..controls (13.9,2) and (16.1,2)..(16.1,2.5);
\draw[thick,->] (14.9,2.1)--(15.1,2.1);
\node at (15,1.8) {$c_1$};
\node at (15,3) {$\partial_1\Sigma$};

\draw[thick] (17.9,2.5)..controls (17.9,2) and (20,2)..(20,2.5);
\draw[thick,->] (18.9,2.1)--(19.1,2.1);
\node at (19,1.8) {$c_2$};
\node at (19,3) {$\partial_2\Sigma$};

 \draw[thick] (18.1,-3).. controls (18.1,-2.9) and (17.9, -2.9).. (17.9,-3)..controls (17.9,-3.1) and (18.1,-3.1).. (18.1,-3);
  \draw[thick] (15.1,-3).. controls (15.1,-2.9) and (14.9, -2.9).. (14.9,-3)..controls (14.9,-3.1) and (15.1,-3.1).. (15.1,-3);
  
  \node at (18,-1.1) {$c_3$};
  \node at (15,-1.1) {$c_4$};
  
    \draw[thick] (18.5,-1.5)..controls (18.5,-1.3) and (17.6,-1.3)..(17.6,-1.5);
  
  \draw[thick] (15.4,-1.5)..controls (15.4,-1.3) and (14.5,-1.3)..(14.5,-1.5);
 
 \draw[<-](18,-1.35)--(18.3,-1.35);
 \draw[<-](14.9,-1.35)--(15.2,-1.35);

\end{tikzpicture}}\]

The following result provides criterion to check wether two line fields are in the same $\MCG(\Sigma)$-orbit.

\begin{theorem}\cite[Theorem 1.2.4]{LekiliPolishchuk}\label{thmLP}
Let $\Sigma$ be a surface with boundary, and punctures and let $\cB$ and $\cG$ as above. 
Let $\eta$ and $\eta'$ be two line fields on $\Sigma$. Then $\eta$ and $\eta'$ are in the same $\MCG(\Sigma)$-orbit if and only if one the following occurs:
\begin{enumerate}

\item (for $g=0$) for any $j=1,\ldots,b$ we have $w_{\eta}(c_i)=w_{\eta'}(c_i)$.

\item (for $g=1$) for any $j=1,\ldots,b$ we have $w_{\eta}(c_i)=w_{\eta'}(c_i)$ and 

$$\gcd \{w_{\eta}(\gamma),w_{\eta}(c)+2,\gamma\in \cG, c\in \cB\}=\gcd \{w_{\eta'}(\gamma),w_{\eta}(c)+2,\gamma\in \cG, c\in \cB\}$$

\item (for $g\geq 2$) for any $j=1,\ldots,b$ we have $w_{\eta}(c_i)=w_{\eta'}(c_i)$ and one the following occurs:
\begin{enumerate}

\item there exist $\gamma$ and $\gamma'$ in $\cG\cup\cB$ such that $w_{\eta}(\gamma)$ and $w_{\eta'}(\gamma')$ are odd, or 

\item for any $\gamma$ in $\cG\cup \cB$, the numbers $w_{\eta}(\gamma)$ and $w_{\eta'}(\gamma)$ are even and there exists an $i$ with $w_{\eta}(c_i)=0 \ \mod 4$, or

\item for any $\gamma$ in $\cG\cup \cB$, the numbers $w_{\eta}(\gamma)$ and $w_{\eta'}(\gamma)$ are even, for any $i=1,\ldots ,b+p$ we have $w_{\eta}(c_i)=2 \ \mod 4$ and 
$$\sum_{i=1}^g(\frac{1}{2}w_{\eta}(\alpha_i)+1)(\frac{1}{2}w_{\eta}(\beta_i)+1)=\sum_{i=1}^g(\frac{1}{2}w_{\eta'}(\alpha_i)+1)(\frac{1}{2}w_{\eta}(\beta_i)+1).$$
\end{enumerate}
\end{enumerate}
\end{theorem}

\subsection{Application to derived equivalences}

Let $S$, $M=M_{\gpoint}\cup M_{\rpoint}$ and $P=P_{\rpoint}$ be as in Section \ref{sect::locallyGentleAlgebras}. Denote by $g$ the genus of $\Sigma$, $b$ the number of boundary components of $S$ and $p$ the cardinal of $P$.   Define the following sets of simple closed curves on $S\setminus P$ as the previous section $\cB=\{c_1,\ldots,c_{b+p}\}$ and $\cG=\{\alpha_1,\beta_1,\ldots,\alpha_g,\beta_g\}$ .
 
For each $j=1,\ldots, b+p$, we denote by $n(j)$ the number of~$\gpoint$ marked points on $\partial_j S$.  

Combining Theorems~\ref{theo::derivedInvariants} and~\ref{thmLP}, we obtain a numerical criterion to decide when two gentle algebras are derived equivalent.

\begin{theorem}\label{thm::LP+main}

Let $A$ and $A'$ be two gentle algebras with associated dissected surfaces  $(S,M,P, \Delta)$ and $(S',M', P', \Delta')$, respectively. Let $\cG=\{\alpha_1,\ldots, \beta_g\}$, $\cB=\{c_1,\ldots,c_{b+p}\}$ (resp. $\cG'=\{\alpha'_1,\ldots, \beta'_{g'}\}$, $\cB'=\{c'_1,\ldots,c'_{b'+p'}\}$ ) subsets of simple closed curves on $S\setminus P$ (resp. $S'\setminus P'$) as before. Then the algebras $A$ and $A'$ are derived equivalent if and only if the following numbers coincide:

\begin{enumerate}

\item $g=g'$, $b=b'$, $\sharp M=\sharp M'$, $\sharp P=\sharp P'$; 

\item there exists a permutation $\sigma\in \mathfrak{S}_{b+p}$ such that $n(\sigma (j))=n'(j)$ and $w^{\Delta^*}(c_{\sigma(j)})=w^{\Delta'^*}(c'_{j})$, for any $j=1,\ldots, b$;
 
 \item for $g=g'\geq 1$ one of the following holds 
 
 \begin{enumerate}
 
 \item for $g=g'=1$,  we have
 
$$\gcd \{w^{\Delta^*}(\gamma),w^{\Delta^*}(c)+2,\gamma\in \cG, c\in \cB\}=\gcd \{w^{\Delta'^*}(\gamma'),w^{\Delta'^*}(c')+2,\gamma'\in \cG', c'\in \cB'\}$$

\item for $g=g'\geq 2$ one the following occurs:
\begin{enumerate}

\item there exist $\gamma\in \cG\cup\cB$ and $\gamma'\in \cG'\cup\cB'$  such that $w^{\Delta^*}(\gamma)$ and $w^{\Delta'^*}(\gamma')$ are odd, or 

\item for any $\gamma\in \cG\cup\cB$ and $\gamma'\in \cG'\cup\cB'$, the numbers $w^{\Delta^*}(\gamma)$ and $w^{\Delta'^*}(\gamma')$ are even and there exists an $i$ with $w^{\Delta^*}(c_i)=0 \ \mod 4$, or

\item for any $\gamma\in \cG\cup\cB$ and $\gamma'\in \cG'\cup\cB'$, the numbers $w^{\Delta^*}(\gamma)$ and $w^{\Delta'^*}(\gamma')$ are even and, for any $i=1,\ldots ,b+p$ we have $w^{\Delta^*}(c_i)=2 \ \mod 4$ and 
$$\sum_{i=1}^g(\frac{1}{2}w^{\Delta^*}(\alpha_i)+1)(\frac{1}{2}w^{\Delta^*}(\beta_i)+1)=\sum_{i=1}^g(\frac{1}{2}w^{\Delta'^*}(\alpha'_i)+1)(\frac{1}{2}w^{\Delta'^*}(\beta'_i)+1) \quad \mod 2$$

\end{enumerate}
\end{enumerate}

\end{enumerate}
\end{theorem}

The first step in the proof of this theorem consists of showing that the different numbers computed from the winding numbers of the curves in $\cG$ above are independent of the choice of the set $\cG$.

In order to prove it we need the following lemma:

\begin{lemma}\label{lemma::quadratic}

In the set up above, assume that for any $\gamma\in \cG\cup \cB$, the winding number $w^{\Delta^*}(\gamma)$ is even. Then there exists a unique quadratic form $q_{\Delta^*}:\Ho_1(S\setminus P,\mathbb Z_2)\to \mathbb Z_2$ satisfying:
\begin{itemize}
\item for all $x$ and $y$ in $\Ho_1(S\setminus P,\mathbb Z_2)$, $$q_{\Delta^*}(x+y)=q_{\Delta^*}(x)+q_{\Delta_*}(y)+(x,y)$$ where $(-,-)$ is the intersection form on $\Ho_1(S\setminus P,\mathbb Z_2)$ and 
\item for any simple closed curve $\gamma$ in $S\setminus P$
$$q_{\Delta^*}([\gamma])=\frac{1}{2}w^{\Delta^*}(\gamma)+1.$$
\end{itemize}
If moreover for all $i=1,\ldots, b+p$ we have $w^{\Delta^*}(c_i)=2\ \mod 4$, the form $q_{\Delta^*}$ descends to a quadratic form $\overline{q}_{\Delta^*}$ on $H_1(\overline{S},\mathbb Z_2)$. Its Arf invariant is given by the formula $${\rm Arf}(\overline{q}_{\Delta^*})=\sum_{i=1}^g \overline{q}_{\Delta^*}([\bar{a_i}])\overline{q}_{\Delta^*}([\bar{b_i}])$$ for any $(\bar{a_1},\bar{b_1},\ldots,\bar{a_g},\bar{b_g})$ symplectic geometric basis of $H_1(\overline{S},\mathbb Z_2)$.

\end{lemma}

\begin{proof}

First note that by Proposition \ref{propLF} (4) and (3), if all winding number of the basis are even, then the winding number of any curve is even, and the line field comes from a vector field. 

The unicity of the map is clear.
The fact that $q_{\Delta^*}$ is well-defined comes from \cite{Johnson} [Theorem 1.A] adpated to the case where the surface has boundary components and punctures (see \cite{LekiliPolishchuk}[Proposition 1.2.2]). 

The equality $w^{\Delta^*}(c_i)=2\  \mod\  4$ is equivalent to $q_{\Delta^*}(c_i)=0$. Hence the second statement is clear.

The computation of the Arf invariant goes back to \cite{Arf}.

\end{proof}

\begin{proof}[Proof of Theorem \ref{thm::LP+main}]

We first show that all numbers associated to the surface $S'$ involved in the statement can be replaced by the image of the curves $\cG$ and $\cB$ through a homeomorphism $S\to S'$.

The item $(1)$ together with the fact that there exists a permutation $\sigma$ with $n(\sigma(j))=n(j)$ is equivalent to the existence of a orientation-preserving homeomorphism $\Phi:S\to S'$ sending $M$ on $M'$ and $P$ to $P'$. 
Moreover for any homeomorphism for any $i=1,\ldots, b+p$, $\Phi(c_i)$ is a curve isotopic to $c'_j$ where $n'(j)=n(i)$, so up to renumbering the boundary components and the punctures of $S'$, we can assume that $\Phi(c_i)=c'_i$ for any $i=1\ldots, b+p$.

Now if the genus $g'$ is $1$, it is proven in \cite{Kawazumi}, that 
$$\gcd (w^{\Delta'^*}(\gamma'),w^{\Delta^*}(c')+2,\gamma\in \cG', c\in \cB'\}=\gcd\{w^{\Delta^*}(\gamma), \gamma \textrm{ non separating closed curve}\}.$$ Therefore this number is independent of the choice of the set $\cG$ and we have $$\gcd (w^{\Delta'^*}(\gamma'),w^{\Delta^*}(c')+2,\gamma\in \cG', c\in \cB'\}=\gcd (w^{\Delta'^*}(\Phi(\gamma)),w^{\Delta'^*}(\Phi(c))+2,\gamma\in \cG, c\in \cB\}.$$

The first item for genus $g'\geq 2$ is equivalent to the fact that the line field $\eta(\Delta'^*)$ does not come from a vector field (see Proposition \ref{propLF} (3) and (4)). Therefore it is equivalent to the fact that there exists a $\gamma\in \cG$ with  $w^{\Delta'^*}(\Phi(\gamma))$ odd.

The second item is also clearly independent of the choice of $\cG$. 

In the case of the item (iii), we have $$\sum_{i=1}^g(\frac{1}{2}w^{\Delta'^*}(\alpha'_i)+1)(\frac{1}{2}w^{\Delta'^*}(\beta'_i)+1)=\sum_{i=1}^g(\frac{1}{2}w^{\Delta'^*}(\Phi(\alpha_i))+1)(\frac{1}{2}w^{\Delta'^*}(\Phi(\beta_i))+1) \quad \mod 2$$ using Lemma \ref{lemma::quadratic}.

\medskip
  
We prove the statement for $g=1$, the proof is similar for $g=0$ and $g\geq 2$.

Suppose that $A$ and $A'$ are derived equivalent. Then by Theorem \ref{theo::derivedInvariants} there exists an orientation preserving homeomorphism $\Phi:S\to S'$ inducing bijections from~$M$ to~$M'$ and from~$P$ to~$P'$ and such that for any $\gamma\in \pi_{1}^{\rm free}(S\backslash P)$,  we have $w^{\Delta^*}(\gamma)=w^{\Delta'^*}(\Phi_*(\gamma))$. 

We can assume that $\Phi$ is a diffeomorphism. Indeed since $S$ and $S'$ are homeomorphic, they are also diffeomorphic. Let $\Psi:S\to S'$ be a diffeomorphism. Then $\Psi^{-1}\circ \Phi$ is a homeomorphism of $S$, and there exists $\bar{\Psi}$ a diffeomorphism of $S$  which is isotopic to $\Psi^{-1}\circ \Phi$. Therefore $\bar{\Psi}$ and $\Psi^{-1}\circ \Phi$ have the same action on the fundamental group. Hence $\Psi\circ\bar{\Psi}$ is a diffeomorphism from $S$ to $S'$ which have the same action as $\Phi$ on the fundamental groups of $S$ and $S'$.

Denote by $\eta$ (resp. $\eta'$) the line fields corresponding to $\Delta$ (resp. $\Delta'$) as defined in Lemma  \ref{lemma Delta eta}. So we have that, for any simple closed curve $\gamma$, 
$$
  w_{\eta}(\gamma)=w^{\Delta^*}(\gamma)=w^{\Delta'^*}(\Phi_*(\gamma))=w_{\eta'}(\Phi_*(\gamma))=w_{\Phi^*(\eta')}(\gamma).
$$ 
Hence the line fields $\eta$ and $\Phi^*(\eta')$ are homotopic by Proposition \ref{propLF}. Moreover, the above equality implies that all the conditions of the theorem are satisfied.

\medskip

Assume now that there exists a homeomorphism~$\Phi:S\to S'$ (which can be again assumed to be a diffeomorphism) satisfying the conditions of the theorem.  As above, denote by~$\eta$ and~$\eta'$ the line fields corresponding to~$\Delta$ and~$\Delta'$, respectively, as defined in Lemma~\ref{lemma Delta eta}.  Then the line field~$\Phi^*(\eta')$ is such that for all~$c\in \cB$, 
\[
 w_{\Phi^*(\eta')}(c) = w^{\Delta'^*}(\Phi_*(c)) = w^{\Delta^*}(c) = w_{\eta}(c).
\]
Moreover, we have that
$$\gcd (w^{\Delta'^*}(\gamma'),w^{\Delta^*}(c')+2,\gamma\in \cG', c\in \cB'\}=\gcd (w^{\Delta'^*}(\Phi(\gamma)),w^{\Delta'^*}(\Phi(c))+2,\gamma\in \cG, c\in \cB\}$$
 and the hypotheses of the theorem translate to the conditions of Theorem~\ref{thmLP} applied to~$\eta$ and~$\Phi^*(\eta)$.  Therefore, by Theorem~\ref{thmLP}, the line fields~$\eta$ and~$\Phi^*(\eta)$ are in the same~$\MCG(S\backslash P)$-orbit.

Let~$\varphi$ a diffeomorphism of $S\setminus P$ be such that~$\varphi^*\Phi^*(\eta')$ is homotopic to~$\eta$.  Then for all~$\gamma \in \pi_1^{\rm free}(S\backslash P)$, we have that
\[
 w^{\Delta^*}(\gamma) = w_{\eta}(\gamma) = w_{\varphi^*\Phi^*(\eta')}(\gamma) = w_{\eta'}(\Phi_*\varphi_*\gamma) = w^{\Delta'^*}(\Phi_*\varphi_*\gamma).
\]
Therefore, the homeomorphism~$\Phi\circ \varphi$ satisfies the hypotheses of Theorem~\ref{theo::derivedInvariants}.  Thus~$A$ and~$A'$ are derived equivalent.
\end{proof}

\begin{remark}
 If one starts with a gentle algebras~$A$, the marked surface as constructed in \cite{OpperPlamondonSchroll} or \cite{PaluPilaudPlamondon2} is given either by a thickening of a ribbon graph or by glueing polygons. Finding the genus and number of boundary components and marked points of the surface present no difficulty (see for instance~\cite[Remark 4.11]{PaluPilaudPlamondon2}), and neither does finding the curves $c_i$ (they are constructed in~\cite{AvellaAlaminosGeiss}), but finding a geometric symplectic basis $\cG$ may be much more complicated in high genus. 
\end{remark}

\section{Reproving known results on gentle algebras}\label{sect::reprovingKnownResults}

\subsection{The class of gentle algebras is stable under derived equivalences}

The following result was first proved by J.~Schr\"oer and A.~Zimmermann in \cite{SchroerZimmermann}.
Their proof relies on the embedding of the bounded derived category of an algebra into the stable module category of its repetitive algebra (see \cite[Section II.2]{Happel}).
We provide a new proof using the geometric model of the bounded derived category of gentle algebras.
\begin{theorem}[\cite{SchroerZimmermann}]\label{theo::gentle-is-derived-closed}
 Let~$A$ and~$B$ be two finite-dimensional~$k$-algebras which are derived-equivalent.
 If~$A$ is gentle, then so is~$B$.
\end{theorem}
\begin{proof} 
 Since~$A$ is gentle, Theorem~\ref{theo::bijectionDissectionsGentleAlgebras} implies that there exists a marked surface~$(S,M,P)$
 and an admissible~$\gpoint$-dissection~$\Delta$ such that~$A \cong A(\Delta)$.
 By a theorem of J.~Rickard \cite{Rickard}, there exists a tilting object~$T$ in~$D^b(\MOD A)$
 whose endomorphism algebra is isomorphic to~$B$.
 By Theorem~\ref{theo::siltingObjects}, the tilting object~$T$ is isomorphic to a direct sum~$\bigoplus_{i=1}^n \P_{(\gamma_i, f_i)}$,
 where~$\{\gamma_1, \ldots, \gamma_n\}$ is an admissible~$\gpoint$-dissection of~$(S,M,P)$.
 Using \cite[Theorem 3.3]{OpperPlamondonSchroll}, one sees that the endomorphism algebra of~$T$ has to be gentle.
 Thus~$B$ is gentle.
\end{proof}

\begin{remark}
 While it is assumed in \cite{SchroerZimmermann} that the field~$k$ is algebraically closed, the proofs in that paper seems to be valid over any field.
 The above argument also works for any field.
\end{remark}

\subsection{Gentle algebras are Gorenstein}

The following result was first proved by~C.Geiss and I.~Reiten \cite{GeissReiten}. 
Recall that a finite-dimensional algebra is~\emph{Iwanaga--Gorenstein} if the projective dimensions of its injective modules
and the injective dimensions of its projective modules are bounded.

We provide a new proof of the following result.
\begin{theorem}[\cite{GeissReiten}]\label{theo::gentle-is-Gorenstein}
 Gentle algebras are Iwanaga--Gorenstein.
\end{theorem}
\begin{proof}
 Since the opposite algebra of a gentle algebra is gentle,
 it suffices to show that the projective dimensions of its injective modules are bounded;
 the fact that the injective dimensions of its projective modules are bounded is then obtained by applying duality.
 
 Let~$A$ be a gentle algebra.
 By Theorem~\ref{theo::bijectionDissectionsGentleAlgebras}, there exists a marked surface~$(S,M,P)$
 and an admissible~$\gpoint$-dissection~$\Delta$ such that~$A \cong A(\Delta)$.
 
 Let~$I$ be an indecomposable injective module.
 Then~$I = \nu P$ for an indecomposable projective module~$P$, where~$\nu$ is the Nakayama functor.
 Note that~$P = \P_{(\gamma,f)}$, where~$\gamma$ is one of the arcs of~$\Delta$.
 By \cite[Theorem 5.1]{OpperPlamondonSchroll}, the Nakayama functor~$\nu$ applied to~$\P_{(\gamma,f)}$ yields an object~$\P_{(\gamma',f')}$,
 where~$\gamma'$ is obtained from~$\gamma$ by moving both its endpoints along their respective boundary components.
 Thus~$\gamma'$ is a finite~$\gpoint$-arc, so $I=\nu P = \P_{(\gamma',f')}$ is in~$K^b(\proj A)$.
 This is equivalent to the statement that~$I$ has finite projective dimension.
\end{proof}

\section{Examples}\label{sect::examples}

Consider the following two gentle algebras given by the following quivers and where dotted arcs indicate zero relations:

\[\scalebox{0.9}{
\begin{tikzpicture}[>=stealth,scale=0.6]
\node (A1) at (0,0) {$\bullet$};
\node (A2) at (2,0) {$\bullet$};
\node (A3) at (4,0) {$\bullet$};

\draw[thick,->] (0.2,0.1)--(1.8,0.1);
\draw[thick,->] (0.2,-0.1)--(1.8,-0.1);
\draw[thick,->] (2.2,0.1)--(3.8,0.1);
\draw[thick,->] (2.2,-0.1)--(3.8,-0.1);

\draw[dotted, thick] (1.5,0.2) arc (180:0:0.5);
\draw[dotted, thick] (1.5,-0.2) arc (180:360:0.5);

\node at (2,-1.5) {$\Lambda_1$};

\begin{scope}[xshift=5cm]
\node (A1) at (0,0) {$1$};
\node (A2) at (2,2) {$2$};
\node (A3) at (4,0) {$3$};

\draw[thick, ->] (A1)--(A2);
\draw[thick, ->] (A2)--(A3);
\draw[thick,<-] (0.2,0.1)--(3.8,0.1);
\draw[thick,<-] (0.2,-0.1)--(3.8,-0.1);

\draw[thick, dotted] (0.8,0.2) arc (0:30:0.8);
\draw[thick, dotted] (1.5,1.5) arc (230:310:0.8);
\draw[thick, dotted] (3.2,-0.2) arc (210:480:0.6);

\node at (2,-1.5) {$\Lambda_2$};
\end{scope}

\end{tikzpicture}}\]

Both corresponds to dissections of a torus with one boundary components with two $\gpoint$ and two $\rpoint$ marked points. 

Here is the corresponding $\gpoint$ dissection $\Delta_1$ for $\Lambda_1$ (where opposite sides of the dotted square are identified), together with its dual $\rpoint$ dissection $\Delta_1^*$.

\[\scalebox{1}{
\begin{tikzpicture}[>=stealth,scale=1]

\draw[dotted] (0,0)--(2,0)--(2,2)--(0,2)--(0,0);

\draw (0,0) circle (0.3);
\draw (2,0) circle (0.3);
\draw (0,2) circle (0.3);
\draw (2,2) circle (0.3);

\draw[red,fill=red] (0,0.3) circle (0.05);
\draw[red,fill=red] (0,-0.3) circle (0.05);
\draw[red,fill=red] (2,0.3) circle (0.05);
\draw[red,fill=red] (2,-0.3) circle (0.05);
\draw[red,fill=red] (2,1.7) circle (0.05);
\draw[red,fill=red] (2,2.3) circle (0.05);
\draw[red,fill=red] (0,1.7) circle (0.05);
\draw[red,fill=red] (0,2.3) circle (0.05);

\draw[dark-green] (0.3,0)--(1.7,0);
\draw[dark-green] (0.3,2)--(1.7,2);
\draw[dark-green] (0.3,0)--(1.7,2);
\draw[dark-green] (0.3,0)--(1.7,0);
\draw[dark-green] (0.3,0).. controls (0.3,0.5) and (-0.3,1.5)..(-0.3,2);
\draw[dark-green] (2.3,0).. controls (2.3,0.5) and (1.7,1.5)..(1.7,2);

\draw[dark-green, fill=white] (0.3,0) circle (0.05);
\draw[dark-green, fill=white] (-0.3,0) circle (0.05);
\draw[dark-green, fill=white] (0.3,2) circle (0.05);
\draw[dark-green, fill=white] (-0.3,2) circle (0.05);
\draw[dark-green, fill=white] (2.3,0) circle (0.05);
\draw[dark-green, fill=white] (1.7,0) circle (0.05);
\draw[dark-green, fill=white] (2.3,2) circle (0.05);
\draw[dark-green, fill=white] (1.7,2) circle (0.05);

\begin{scope}[xshift=5cm]

\draw (0,0) circle (0.3);
\draw (2,0) circle (0.3);
\draw (0,2) circle (0.3);
\draw (2,2) circle (0.3);

\draw[red,fill=red] (0,0.3) circle (0.05);
\draw[red,fill=red] (0,-0.3) circle (0.05);
\draw[red,fill=red] (2,0.3) circle (0.05);
\draw[red,fill=red] (2,-0.3) circle (0.05);
\draw[red,fill=red] (2,1.7) circle (0.05);
\draw[red,fill=red] (2,2.3) circle (0.05);
\draw[red,fill=red] (0,1.7) circle (0.05);
\draw[red,fill=red] (0,2.3) circle (0.05);

\draw[red] (0,0.3)--(0,1.7);
\draw[red] (2,0.3)--(2,1.7);
\draw[red] (2,0.3)--(0,1.7);
\draw[red] (0,-0.3).. controls (0.5,-0.3) and (1.5,0.3)..(2,0.3);
\draw[red] (0,1.7).. controls (0.5,1.7) and (1.5,2.3)..(2,2.3);

\draw[dark-green, fill=white] (0.3,0) circle (0.05);
\draw[dark-green, fill=white] (-0.3,0) circle (0.05);
\draw[dark-green, fill=white] (0.3,2) circle (0.05);
\draw[dark-green, fill=white] (-0.3,2) circle (0.05);
\draw[dark-green, fill=white] (2.3,0) circle (0.05);
\draw[dark-green, fill=white] (1.7,0) circle (0.05);
\draw[dark-green, fill=white] (2.3,2) circle (0.05);
\draw[dark-green, fill=white] (1.7,2) circle (0.05);

\draw[blue, ->] (-0.5,0.8)--(2.5,0.8);
\draw[blue, ->] (0.8,-0.5)--(0.8,2.5);

\draw[blue,<-] (0,0.5) arc (90:-20:0.5);
\draw[blue,->] (2,0.5) arc (90:160:0.5);
\draw[blue,<-] (2,1.5) arc (270:160:0.5);
\draw[blue,->] (0,1.5) arc (270:340:0.5);

\node[blue] at (0.5,0.5) {$c$};
\node[blue] at (2.2,1) {$\alpha$};
\node[blue] at (1,2.2) {$\beta$};

\end{scope}

\end{tikzpicture}}\]

One easily computes the winding numbers of $w^{\Delta^*_1}(c)=-2$, $w^{\Delta^*_1}(\alpha)=0$ and $w^{\Delta^*_1}(\beta)=0$. Hence the greatest common divisor of Theorem \ref{thm::LP+main} is $0$.

For $\Lambda_2$ one gets the following picture for $\Delta_2$ and $\Delta_2^*$.

\[\scalebox{1}{
\begin{tikzpicture}[>=stealth,scale=1]

\draw (0,0) circle (0.3);
\draw (2,0) circle (0.3);
\draw (0,2) circle (0.3);
\draw (2,2) circle (0.3);

\draw[red,fill=red] (0,0.3) circle (0.05);
\draw[red,fill=red] (0,-0.3) circle (0.05);
\draw[red,fill=red] (2,0.3) circle (0.05);
\draw[red,fill=red] (2,-0.3) circle (0.05);
\draw[red,fill=red] (2,1.7) circle (0.05);
\draw[red,fill=red] (2,2.3) circle (0.05);
\draw[red,fill=red] (0,1.7) circle (0.05);
\draw[red,fill=red] (0,2.3) circle (0.05);

\draw[red] (0,0.3).. controls (0,0.3) and (1,2.3)..(2,2.3);
\draw[red] (2,2.3).. controls (1.2,2.3) and (2,0.5)..(2,0.3);
\draw[red] (0,2.3).. controls (-0.8,2.3) and (0,0.5)..(0,0.3);
\draw[red] (0,-0.3).. controls (0.5,-0.3) and (1.5,0.3)..(2,0.3);
\draw[red] (0,1.7).. controls (0.5,1.7) and (1.5,2.3)..(2,2.3);

\draw[dark-green] (0.3,0)--(1.7,2);
\draw[dark-green] (-0.3,2).. controls (-0.3,1) and (0.3,1).. (0.3,2);
\draw[dark-green] (0.3,0).. controls (0.3,1) and (-0.3,1)..(-0.3,2);
\draw[dark-green] (2.3,0).. controls (2.3,1) and (1.7,1)..(1.7,2);

\draw[dark-green, fill=white] (0.3,0) circle (0.05);
\draw[dark-green, fill=white] (-0.3,0) circle (0.05);
\draw[dark-green, fill=white] (0.3,2) circle (0.05);
\draw[dark-green, fill=white] (-0.3,2) circle (0.05);
\draw[dark-green, fill=white] (2.3,0) circle (0.05);
\draw[dark-green, fill=white] (1.7,0) circle (0.05);
\draw[dark-green, fill=white] (2.3,2) circle (0.05);
\draw[dark-green, fill=white] (1.7,2) circle (0.05);

\end{tikzpicture}}\]
 One easily computes that $w^{\Delta^*_2}(c)=-2=w^{\Delta_1^*}$, hence the algebras have the same AG-invariant. However, we have $w^{\Delta^*_2}(\alpha)=0$ and $w^{\Delta^*_2}(\beta)=2$ hence the greatest common divisor is $2\neq 0$ and the algebras $\Lambda_1$ and $\Lambda_2$ are not derived equivalent.
Note that this was already shown in \cite{Kalck}.

 Consider now the following two algebras:
 \[\scalebox{0.8}{
\begin{tikzpicture}[>=stealth,scale=1.2]

\node (A3) at (0,2) {$\bullet$};
\node (A5) at (1,1) {$\bullet$};
\node (A1) at (2,0) {$\bullet$};
\node (A2) at (3,1) {$\bullet$};
\node (A4) at (2,2) {$\bullet$};
\node (A6) at (4,2) {$\bullet$};

\draw[thick, ->] (A5)--(A1);
\draw[thick, ->] (A2)--(A5);
\draw[thick, ->] (A6)--(A2);
\draw[thick, ->] (A6)--(A4);
\draw[thick, ->] (A4)--(A3);
\draw[thick, ->] (A5)--(A4);
\draw[thick, ->] (A3)--(A5);
\draw[thick, ->] (3,0.6)--(2.4,0);
\draw[thick, ->] (2.2,0.2)--(2.8,0.8);

\draw[thick, dotted] (1.3,1.3) arc (45:135:0.4);
\draw[thick, dotted] (0.5,2) arc (0:-45:0.5);
\draw[thick, dotted] (1.5,2) arc (180:225:0.5);
\draw[thick, dotted] (1.5,1) arc (0:-45:0.5);
\draw[thick, dotted] (2.2,0.3) arc (55:125:0.45);
\draw[dotted, thick] (2.8,0.5)--(2.65,0.65);
\draw[dotted, thick] (2.5,1) arc (180:45:0.5);

\node at (2,-1) {$\Lambda_1$};

\begin{scope}[xshift=5cm]

\node (A6) at (0,0) {$\bullet$};
\node (A1) at (0,3) {$\bullet$};
\node (A4) at (1,1) {$\bullet$};
\node (A3) at (2,0) {$\bullet$};
\node (A2) at (2,3) {$\bullet$};
\node (A5) at (1,2) {$\bullet$};

\draw[thick, ->] (A1)--(A6);
\draw[thick, ->] (A2)--(A1);
\draw[thick, ->] (A5)--(A2);
\draw[thick, ->] (A5)--(A4);
\draw[thick, ->] (A2)--(A3);
\draw[thick, ->] (A4)--(A6);
\draw[thick, ->] (A3)--(A4);
\draw[thick, ->] (A6)--(A3);
\draw[thick, ->] (A1)--(A5);

\draw[thick, dotted] (0.5,0) arc (0:45:0.5);
\draw[thick, dotted] (1.5,0) arc (180:135:0.5);
\draw[thick, dotted] (1,1.5) arc (90:225:0.5);
\draw[thick, dotted] (1.3,2.3) arc (45:135:0.4);
\draw[thick, dotted] (0.5,3) arc (0:-45:0.5);
\draw[thick, dotted] (1.5,3) arc (180:225:0.5);

\node at (1,-1) {$\Lambda_2$};
\end{scope}

\end{tikzpicture}}\]

One can check that these two algebras both come from a dissection of a torus with two boundary components, 3 $\gpoint$ marked points (hence 3 $\rpoint$ marked points), and one $\rpoint$ puncture. 

The corresponding dual dissections $\Delta_1^*$ and $\Delta_2^*$ are as follows:

\[\scalebox{0.8}{
\begin{tikzpicture}[>=stealth,scale=1.2]

\draw[dotted] (0,0)--(2,0)--(2,2)--(0,2)--(0,0);

\draw (0,0) circle (0.3);
\draw (2,0) circle (0.3);
\draw (0,2) circle (0.3);
\draw (2,2) circle (0.3);

\draw[red,fill=red] (0,0.3) circle (0.05);
\draw[red,fill=red] (0,-0.3) circle (0.05);
\draw[red,fill=red] (2,0.3) circle (0.05);
\draw[red,fill=red] (2,-0.3) circle (0.05);
\draw[red,fill=red] (2,1.7) circle (0.05);
\draw[red,fill=red] (2,2.3) circle (0.05);
\draw[red,fill=red] (0,1.7) circle (0.05);
\draw[red,fill=red] (0,2.3) circle (0.05);

\draw[red, fill=red] (0.5,1) circle (0.05);
\draw (1,0.5) circle (0.2);
\draw[red, fill=red] (1,0.7) circle (0.05);

\draw[red] (0,-0.3)--(2,-0.3);
\draw[red] (0,1.7)--(2,1.7);
\draw[red] (0,-0.3).. controls (-1,-0.3) and (0,1.7)..(0,1.7);
\draw[red] (2,-0.3).. controls (1,-0.3) and (2,1.7)..(2,1.7);
\draw[red] (1,0.7)..controls (0.5,0.7) and (1,-0.3)..(0,-0.3);

\draw[red] (0,0.3)--(0.5,1);
\draw[red] (0,1.7)--(0.5,1);
\draw[red] (1,0.7)--(0.5,1);

\draw[dark-green, fill=white] (0.3,0) circle (0.05);
\draw[dark-green, fill=white] (-0.3,0) circle (0.05);
\draw[dark-green, fill=white] (0.3,2) circle (0.05);
\draw[dark-green, fill=white] (-0.3,2) circle (0.05);
\draw[dark-green, fill=white] (2.3,0) circle (0.05);
\draw[dark-green, fill=white] (1.7,0) circle (0.05);
\draw[dark-green, fill=white] (2.3,2) circle (0.05);
\draw[dark-green, fill=white] (1.7,2) circle (0.05);
\draw[dark-green, fill=white] (1,0.3) circle (0.05);

\draw[blue,->] (-0.5,1.25)--(2.5,1.25);
\draw[blue, ->] (1.5,-0.5)--(1.5,2.5);

\node[blue] at (2.2,1.4) {$\alpha$};
\node[blue] at (1.3,2.2) {$\beta$};

\node at (1,-1) {$\Delta_1^*$};

\begin{scope}[xshift=5cm]

\draw[dotted] (0,0)--(2,0)--(2,2)--(0,2)--(0,0);

\draw (0,0) circle (0.3);
\draw (2,0) circle (0.3);
\draw (0,2) circle (0.3);
\draw (2,2) circle (0.3);

\draw[red,fill=red] (0,0.3) circle (0.05);
\draw[red,fill=red] (0,-0.3) circle (0.05);
\draw[red,fill=red] (2,0.3) circle (0.05);
\draw[red,fill=red] (2,-0.3) circle (0.05);
\draw[red,fill=red] (2,1.7) circle (0.05);
\draw[red,fill=red] (2,2.3) circle (0.05);
\draw[red,fill=red] (0,1.7) circle (0.05);
\draw[red,fill=red] (0,2.3) circle (0.05);

\draw[red, fill=red] (0,1) circle (0.05);
\draw[red, fill=red] (2,1) circle (0.05);

\draw (1,0.8) circle (0.2);
\draw[red, fill=red] (1,1) circle (0.05);

\draw[red] (0,-0.3)--(2,-0.3);
\draw[red] (0,1.7)--(2,1.7);

\draw[red] (0,0.3)--(0,1);
\draw[red] (1,1)--(0,1);

\draw[red] (0,0.3)..controls (1,0.3) and (2,1)..(2,1.7);
\draw[red] (0,-0.3)..controls (1,-0.3) and (1.5,1)..(2,1);
\draw[red] (1,1)--(2,1.7);
\draw[red] (2,1)--(2,0.3);

\draw[dark-green, fill=white] (0.3,0) circle (0.05);
\draw[dark-green, fill=white] (-0.3,0) circle (0.05);
\draw[dark-green, fill=white] (0.3,2) circle (0.05);
\draw[dark-green, fill=white] (-0.3,2) circle (0.05);
\draw[dark-green, fill=white] (2.3,0) circle (0.05);
\draw[dark-green, fill=white] (1.7,0) circle (0.05);
\draw[dark-green, fill=white] (2.3,2) circle (0.05);
\draw[dark-green, fill=white] (1.7,2) circle (0.05);
\draw[dark-green, fill=white] (1,0.6) circle (0.05);

\draw[blue,->] (-0.5,1.25)--(2.5,1.25);
\draw[blue, ->] (1.5,-0.5)--(1.5,2.5);

\node[blue] at (2.2,1.4) {$\alpha$};
\node[blue] at (1.3,2.2) {$\beta$};

\node at (1,-1) {$\Delta_1^*$};

\end{scope}

\end{tikzpicture}}\]

One computes that for both algebras the winding numbers of the curves $c_1$, $c_2$ and $c_3$ are respectively $-3$, $0$ and $-3$. Therefore the algebras $\Lambda_1$ and $\Lambda_2$ have the same AG-invariant. Moreover, since $-3+2=1$, both gcd are $1$  and so the algebras are derived equivalent.

\section{Complete derived invariant for surface cut algebras}\label{sect::surface-cut-algebras}

A theorem similar to Theorem \ref{theo::derivedInvariants} has been shown in \cite{AmiotGrimeland} for  \emph{surface cut algebras}. These algebras are gentle  and appear in the theory of cluster categories \cite{DavidRoeslerSchiffler, AmiotGrimeland}. In this context, they naturally arise with a slightly different model from the one in \cite{OpperPlamondonSchroll}.  The purpose of this Section is to give a reformulation  of Theorem \ref{thm::LP+main} in terms of this geometric model. 

Let $\Sigma$ be an oriented connected compact surface of genus $g$, with $b$ boundary components and $\cM$ be a finite set of marked points on the boundary of $\Sigma$, such that there is at least one marked point on each component of the boundary.  By an \emph{arc}, we mean the isotopy class of a simple curve with endpoints being in $\cM$ that does not cut out a monogon or a digon of the surface. Two arcs are called \emph{non intersecting} if there exist representatives in their respective homotopy class that do not intersect. A \emph{triangulation} of $(\Sigma,\cM)$ is a maximal collection of non intersecting arcs.

To each triangulation $\Delta$ it is possible to associate a quiver $Q^\Delta$ as follows (see \cite{FominShapiroThurston}): the set of vertices $Q^\Delta_0$ of $Q^\Delta$ is the set of internal arcs of $\Delta$, and the set of arrows $Q_1^\Delta$ is in bijection with the set of oriented internal angles between two arcs of $\Delta$. We also denote by $Q^\Delta_2$ the set of oriented internal triangles of $\Delta$, that are the triangles whose three sides are internal arcs of the triangulation. The Jacobian algebra associated to $\Delta$ has been introduced in \cite{AssemBrustleCharbonneauJodoinPlamondon}. It is defined as the quotient $$ {\rm Jac} (\Delta):= kQ^\Delta/ \langle \alpha\beta,\ \beta\gamma,\ \gamma\alpha,\  \forall \alpha\beta\gamma\in Q_2\rangle.$$

\begin{definition}\cite[Def. 2.8]{AmiotGrimeland}\label{def admissible cut}
An \emph{admissible cut} of $\Delta$ is map $d:Q_1\to \{0,1\}$ such that 
\begin{itemize}
\item for each internal triangle $\alpha\beta\gamma$ in $Q_2$, we have $d(\alpha)+d(\beta)+d(\gamma)=1$ 
\item for each angle $\alpha$ in a triangle which is not internal, we have $d(\alpha)=0$ . 
\end{itemize}

An admissible cut defines a positive $\mathbb Z$-grading on the Jacobian algebra ${\rm Jac}(\Delta)$. The \emph{surface cut algebra} (first introduced in \cite{DavidRoeslerSchiffler}) $\Lambda:=(\Delta,d)$  associated to $\Delta$ and $d$, is the degree zero subalgebra of ${\rm Jac}(\Delta,d)$. In other words $\Lambda:=\Lambda(\Delta,d)$ is the path algebra of degree zero arrows of $Q^\Delta$ divided by the ideal generated by the relation $\alpha\beta=0$ as soon as the arrows $\alpha$ and $\beta$ correspond to angles of the same triangle.  
 
\end{definition}

\begin{remark}
One can show that the surface cut algebras are exactly the gentle algebras of global dimension $\leq 2$ such that the functor $-\otimes_{\Lambda} \Ext{2}{\Lambda}(D\Lambda,\Lambda)$ is nilpotent (that is $\tau_2$-finite). For example, the algebra $\Lambda_1$ given in the first example of Section \ref{sect::examples} has global dimension $2$, but the associated Jacobian algebra is an infinite dimensional algebra. It comes from a triangulation of a surface with empty boundary, so it is not a surface cut algebra in the sense of \cite{AmiotGrimeland}. 
\end{remark}
An admissible cut $d$ of a triangulation $\Delta$ defines a map $d:\pi_1^{\rm free}(\Sigma)\to \mathbb Z$ that counts with signs the number of times the loop $\gamma$ intersects an angle of degree $1$ of $\Delta$ (see \cite[section 2]{Amiot} for details). 

In fact this map $d$ has a geometric interpretation.

\begin{lemma}\label{lemma d eta}
Let $(\Sigma,\cM)$ be a smooth orientable surface with marked points, and let $(\Delta,d)$ be a triangulation of $(\Sigma,\cM)$ together with an admissible cut.  Then there exists a unique element in $\LF(\Sigma)$ such that for any element $\gamma\in \pi_1^{\rm free}(\Sigma)$, we have the equality $w_{\eta}(\gamma)=d(\gamma)$.

\end{lemma}

\begin{proof}

Choose a smooth representatives of each arc of $\Delta$. We define $\eta$ on each triangle of $\Delta$ so that it is the tangent field along each internal arc of $\Delta$. The corresponding foliation is drawn in the following picture depending on wether there is one, two or three sides of the triangle that are internal arcs. This will define a line field on $\Sigma\backslash \cM$.

\[\scalebox{0.9}{
\begin{tikzpicture}[>=stealth,scale=0.6]

\draw[gray,fill=gray] (-2.4,0)--(-2,0)--(0,4)--(2,0)--(2.4,0)--(0,4.6)--(-2.4,0);

\draw (-2,0)--(2,0)--(0,4)--(-2,0);
\draw[blue!10, fill=blue!10] (-2,-0.6)--(-2,0)--(0,4)--(2,0)--(2,-0.6)--(-2,-0.6);
\draw[thick] (-2,0)--(2,0);
\node at (-2,0) {$\bullet$};
\node at (0,4) {$\bullet$};
\node at (2,0) {$\bullet$};

\draw[blue] (-1.75,0.5)--(1.75,0.5);
\draw[blue] (-1.5,1)--(1.5,1);
\draw[blue] (-1.25,1.5)--(1.25,1.5);
\draw[blue] (-1,2)--(1,2);
\draw[blue] (-0.75,2.5)--(0.75,2.5);
\draw[blue] (-0.5,3)--(0.5,3);
\draw[blue] (-0.25,3.5)--(0.25,3.5);

\node at (0,-2) { two boundary sides};

\begin{scope}[xshift=8cm]

\draw[blue!10,fill=blue!10] (-2.4,0)--(2.4,0)--(0.4,4)--(-0.4,4)--(-2.4,0);
\draw[gray,fill=gray] (-0.4,4)--(0,4.6)--(0.4,4)--(-0.4,4);

\draw (-2,0)--(2,0)--node (B1){} (0,4)--node (B2){}(-2,0);
\draw[gray, fill=gray] (-2.4,-0.6)--(-2.4,0)--(2.4,0)--(2.4,-0.6)--(-2.4,-0.6);
\draw[thick] (-2,0)--(0,4)--(2,0);
\node at (-2,0) {$\bullet$};
\node at (0,4) {$\bullet$};
\node at (2,0) {$\bullet$};

\draw[blue] (0,4)--(-1.75,0);
\draw[blue] (0,4)--(-1.25,0);
\draw[blue] (0,4)--(-0.75,0);
\draw[blue] (0,4)--(-0.25,0);
\draw[blue] (0,4)--(0.25,0);
\draw[blue] (0,4)--(0.75,0);
\draw[blue] (0,4)--(1.25,0);
\draw[blue] (0,4)--(1.75,0);

\draw[->,thick](B1)--(B2);

\node at (0,-2) {one boundary side};

\end{scope}

\begin{scope}[xshift=16cm]

\draw[blue!10,fill=blue!10] (-2.75,-0.5)--(0,5)--(2.75,-0.5)--(-2.75,-0.5);

\draw[gray, fill=gray] (-2.75,-0.5)--(-1.75,-0.5)--(-2.25,0.5)--(-2.75,-0.5);
\draw[gray, fill=gray] (2.75,-0.5)--(1.75,-0.5)--(2.25,0.5)--(2.75,-0.5);
\draw[gray, fill=gray] (-0.5,4)--(0.5,4)--(0,5)--(-0.5,4);

\draw (-2,0)--node (A1){}(2,0)--node (A2){}(0,4)--node (A3){}(-2,0);

\draw[thick] (-2,0)--(0,4)--(2,0);
\node at (-2,0) {$\bullet$};
\node at (0,4) {$\bullet$};
\node at (2,0) {$\bullet$};

\draw[blue] (-2,0)..controls (-1,0.5) and (1,0.5)..(2,0);
\draw[blue] (-2,0)..controls (-1,1) and (1,1)..(2,0);
\draw[blue] (-2,0)..controls (-1,1.5) and (1,1.5)..(2,0);
\draw[blue] (-2,0)..controls (-1,2) and (1,2)..(2,0);
\draw[blue] (-2,0)..controls (-0.5,3) and (0.5,3)..(2,0);
\draw[blue] (-2,0)..controls (0,4) and (0,4)..(2,0);

\node at (0,-2) { no boundary side and $d(\alpha)=1$};

\draw[->,thick] (A2)--node [fill=blue!10, inner sep=0pt]{$\alpha$}(A3);
\draw[->,thick] (A3)--(A1);
\draw[->,thick] (A1)--(A2);

\end{scope}

\end{tikzpicture}}\]
It is then clear that the degree of any arrow of $Q^\Delta$ is equal to the winding number of the corresponding curve in the triangle. The proof is then similar to the one of Lemma \ref{lemm::computingWindingNumber}. 

\end{proof}

The following result  is the main theorem in \cite{AmiotGrimeland}.

\begin{theorem}\label{thmAG}\cite[Thm 3.13]{AmiotGrimeland}
Let $\Lambda=(\Sigma,\cM,\Delta,d)$ and $\Lambda'=(\Sigma',\cM',\Delta',d')$ be surface cut algebras.The the following are equivalent
\begin{enumerate}
\item $\Lambda$ and $\Lambda'$ are derived equivalent;
\item There exists an orientation-preserving homeomorphism $\Phi:\Sigma\to\Sigma'$ with $\Phi(\cM)=\cM'$ such that for any loop $\gamma\in \pi_1^{\rm free}(\Sigma)$ we have $d(\gamma)=d'(\Phi\circ\gamma)$.
\end{enumerate}
\end{theorem}

Combining this theorem  with Lemma \ref{lemma d eta}, we obtain a result analogous to Theorem \ref{theo::derivedInvariants}. However, it is interesting to note that the proof of Theorem \ref{thmAG} uses completely different tools. The main ingredient in the proof of the above result are (graded) cluster mutations and cluster categorification. 

Using Theorem \ref{thmLP}, it is possible to translate this result into a numerical criterion to check wether two surface cut algebras are derived equivalent.

Denote by $\cG=\{\alpha_1,\beta_1,\ldots,\alpha_g,\beta_g\}$ and $\cB=\{c_1,\ldots,c_b\}$ a basis of the fundamental group as in section \ref{sect::MCG}. For each $j=1,\ldots, b$, we denote by $n(j)=\sharp (\cM\cap \partial_j\Sigma)$ the number of marked points on $\partial_j\Sigma$.

\begin{theorem}
Let $\Lambda=(\Sigma,\cM,\Delta,d)$ and $\Lambda'=(\Sigma',\cM',\Delta',d')$ be surface cut algebras. Let $\cB$ and $\cG$ (resp. $\cB'$ and $\cG'$) be sets of simple closed curves on $\Sigma$ (resp. $\Sigma'$).
Then $\Lambda$ and $\Lambda'$ are derived equivalent if and only if 
the following numbers coincide

\begin{enumerate}
\item $g=g'$, $b=b'$, $\sharp \cM=\sharp \cM'$;
\item there exists a permutation $\sigma\in \mathfrak{S}_b$ such that for any $i=1,\ldots,b$ we have  $n(\sigma(i))=n'(i)$ and $d(c_{\sigma(i)})=d'(c'_{i})$;
\item and if the genus~$g(\Sigma)$ is $\geq 1$, one of the following holds:
\begin{enumerate} 

\item for $g(\Sigma)=1$, we have 

$$\gcd\{d(\gamma),d(c)+2,\gamma\in \cG, c\in \cB\}=\gcd\{d'(\gamma'),d'(c')+2,\gamma'\in \cG', c'\in \cB'\}$$
\item for $g(\Sigma)\geq 2$,  one the following occurs

\begin{enumerate}
\item there exist $\gamma$  in $\cG\cup \cB$ and $\gamma'$ in $\cG'\cup \cB'$ such that $d(\gamma)$ and $d'(\gamma')$ are odd, or 
\item for any $\gamma$ in $\cG\cup \cB$, for any $\gamma'$ in $\cG'\cup \cB$, the numbers $d(\gamma)$ and $d'(\Phi(\gamma))$ are even and there exists an $i$ with $d(c_i)=0\  \mod 4$, or
\item for any $\gamma$ in $\cG\cup \cB$ and $\gamma\in \cG\cup \cB$, the numbers $d(\gamma)$ and $d'(\gamma')$ are even, for any $i=1,\ldots ,b$ we have $d(c_i)=2\ \mod 4$ and 
$$\sum_{i=1}^g\frac{1}{2}(d(\alpha_i)+1)(d(\beta_i)+1)=\sum_{i=1}^g\frac{1}{2}(d'(\alpha'_i)+1)(d'(\beta'_i)+1)\ \mod 2.$$ 
\end{enumerate}

\end{enumerate}
\end{enumerate}
\end{theorem}

\begin{remark}The case $g=0$ was already treated in \cite{AmiotGrimeland}, and the case $g=1$ and $b=1$ was stated in \cite{Amiot}. In this case the degree of the unique curve $c \in \cB$ is always $-2$ (see Proposition 2.9 in \cite{AmiotGrimeland}), and so $\gcd\{d(\alpha),d(\beta), d(c)+2\}=\gcd (d(\alpha), d(\beta))$ and we recover Theorem 3.1 in \cite{Amiot}. 
\end{remark}

We end this section with  an example.

Let $(\Sigma,\cM)$ be a surface of genus $2$ with  one boundary component and one marked point. A symplectic basis is drawn in blue.

\[\scalebox{0.8}{
\begin{tikzpicture}[>=stealth,scale=0.5]

\draw[fill=blue!10] (0,0)--node [rotate=90]{$>$} (0,3)--node [rotate=45]{$>\!\!>$}(2,5)--node {$<$}(5,5)--node[rotate=135]{$>\!\!>$}(7,3)--node {$\triangle$}(7,0)--node [rotate=45]{$>\!\!\! |$}(5,-2)--node [rotate=90]{$\triangle$}(2,-2)--node [rotate=135]{$>\!\!\! |$}(0,0);

\node at (0,3) {$\bullet$};
\node at (2,5) {$\bullet$};\node at (5,5) {$\bullet$};
\node at (7,3) {$\bullet$};\node at (7,0) {$\bullet$};
\node at (5,-2) {$\bullet$};\node at (2,-2) {$\bullet$};

\draw[fill=white] (0,0)..controls (2,2) and (2,-1)..(0,0);
\node at (0,0) {$\bullet$};

\draw[blue] (0,2)--node [rotate=45]{$>$}(3,5);
\draw[blue] (6,4)--node {$<$}(1,4);

\draw[blue] (1,-1)-- node {$>$}(6,-1);
\draw[blue] (7,1)--node[rotate=45]{$<$}(4,-2);

\draw[magenta] (0,0.5)..controls (2,2) and (3,-0.5)..node [rotate=-47]{$>$}(0.5,-0.5);

\draw[magenta] (0,2.5) arc (-90:45:0.5);
\draw[magenta] (2.5,5) arc (0:-135:0.5);
\draw[magenta] (4.5,5) arc (-180:-45:0.5);
\draw[magenta] (7,2.5) arc (-90:-225:0.5);
\draw[magenta] (7,0.5) arc (90:225:0.5);
\draw[magenta] (4.5,-2) arc (180:45:0.5);
\draw[magenta] (2.5,-2) arc (0:135:0.5);

\end{tikzpicture}}\]

Consider the following triangulation with cuts, where the cuts are indicated in blue. The quiver of the corresponding algebra is given in the picture to the right. The ideal of relations is generated by the composition of two arrows belonging to the same triangle.

\[\scalebox{1}{
\begin{tikzpicture}[>=stealth,scale=0.5]

\draw[fill=blue!10] (0,0)--(0,3)--(2,5)--(5,5)--(7,3)--(7,0)--(5,-2)--(2,-2)--(0,0);

\node at (0,3) {$\bullet$};
\node at (2,5) {$\bullet$};\node at (5,5) {$\bullet$};
\node at (7,3) {$\bullet$};\node at (7,0) {$\bullet$};
\node at (5,-2) {$\bullet$};\node at (2,-2) {$\bullet$};

\draw[fill=white] (0,0)..controls (2,2) and (2,-1)..(0,0);
\node at (0,0) {$\bullet$};

\draw (0,0)--(2,5);
\draw (0,0)--(5,5);
\draw (0,0)..controls (2,2) and (7,3)..(7,3);

\draw (0,0)..controls (2,-1.5) and (7,3)..(7,3);
\draw (2,-2)--(7,3);
\draw (5,-2)--(7,3);

\draw[blue] (0,1.5) arc (90:70:1.5);
\draw[blue] (2.5,5) arc (0:-110:0.5);
\draw[blue] (6,4) arc (135:195:1.4);
\draw[blue] (3,-2) arc (0:135:1);
\draw[blue] (6,-1) arc (45:70:1.4);

\node at (3.5,-3) {$\Lambda_0= (\Delta_0,d_0)$};

\begin{scope}[xshift=10cm]

\draw[dotted] (0,0)--node (A1)[fill=white, inner sep=0pt] {$1$} (0,3)--node (A2)[fill=white, inner sep=0pt] {$2$} (2,5)--node (B1) [fill=white, inner sep=0pt] {$1$} (5,5)--node (B2) [fill=white, inner sep=0pt] {$2$} (7,3)--node (A3)[fill=white, inner sep=0pt] {$3$} (7,0)--node (A4)[fill=white, inner sep=0pt] {$4$} (5,-2)--node (B3) [fill=white, inner sep=0pt] {$3$} (2,-2)--node (B4)[fill=white, inner sep=0pt] {$4$} (0,0);

\draw[dotted] (0,0)--node (A5)[fill=white, inner sep=0pt] {$5$}(2,5);
\draw[dotted] (0,0)--node (A6)[fill=white, inner sep=0pt] {$6$}(5,5);
\draw[dotted] (0,0)..controls (2,2) and (7,3)..node (A7)[fill=white, inner sep=0pt] {$7$}(7,3);

\draw[dotted] (0,0)..controls (2,-1.5) and (7,3)..node (A8)[xshift=-1cm,yshift=-0.5cm,fill=white, inner sep=0pt] {$8$}(7,3);
\draw[dotted] (2,-2)--node (A9)[fill=white, inner sep=0pt] {$9$}(7,3);
\draw[dotted] (5,-2)--node (A10)[fill=white, inner sep=0pt] {$10$}(7,3);

\draw[thick,->] (A1)--(A2);
\draw[thick,->] (A2)--(A5);
\draw[thick,->] (B1)--(A6);
\draw[thick,->] (A6)--(A5);
\draw[thick,->] (A7)--(A6);
\draw[thick,->] (A6)--(B2);
\draw[thick,->] (A7)--(A8);
\draw[thick,->] (B4)--(A8);
\draw[thick,->] (A8)--(A9);
\draw[thick,->] (A9)--(A10);
\draw[thick,->] (A10)--(B3);
\draw[thick,->] (A10)--(A3);
\draw[thick,->] (A3)--(A4);

\node at (3.5,-3) {$\Lambda_0$};

\end{scope}

\end{tikzpicture}}\]

One then easily computes the degree of each generator of the fundamental group and gets:

\[d_0(c)=-6,\ d_0(\alpha_1)=0,\ d_0(\beta_1)=0, \ d_0(\alpha_2)=2\ \textrm{and }d_0(\beta_2)=0.\]

All the numbers are even, and $d_0(c)=2 \ \mod 4$ so we are in the case (3)(c) of the theorem. The derived equivalence class of this algebra is determined by the Arf invariant, which is $\rm{Arf}(d_0)=1$ in this case.

Now take the algebras given by the following two triangulations with cuts:
\[\scalebox{1}{
\begin{tikzpicture}[>=stealth,scale=0.5]

\draw[fill=blue!10] (0,0)--(0,3)--(2,5)--(5,5)--(7,3)--(7,0)--(5,-2)--(2,-2)--(0,0);

\node at (0,3) {$\bullet$};
\node at (2,5) {$\bullet$};\node at (5,5) {$\bullet$};
\node at (7,3) {$\bullet$};\node at (7,0) {$\bullet$};
\node at (5,-2) {$\bullet$};\node at (2,-2) {$\bullet$};

\draw (0,0)--(2,5);
\draw (0,0)--(5,5);
\draw (0,0)..controls (2,2) and (7,3)..(7,3);

\draw (0,0)..controls (2,-1.5) and (7,3)..(7,3);
\draw (2,-2)--(7,3);
\draw (5,-2)--(7,3);

\draw[fill=white] (0,0)..controls (2,2) and (2,-1)..(0,0);
\node at (0,0) {$\bullet$};

\draw[blue] (0,1.5) arc (90:70:1.5);
\draw[blue] (2.5,5) arc (0:-110:0.5);
\draw[blue] (6,4) arc (135:195:1.4);
\draw[blue] (1,-1) arc (-45:-10:1.4);
\draw[blue] (4,-2) arc (180:70:1);
\draw[blue] (7,1.5) arc (-90:-110:1.5);

\node at (3.5,-3) {$\Lambda_1=(\Delta_0,d_1)$};

\begin{scope}[xshift=10cm]

\draw[fill=blue!10] (0,0)--(0,3)--(2,5)--(5,5)--(7,3)--(7,0)--(5,-2)--(2,-2)--(0,0);

\node at (0,3) {$\bullet$};
\node at (2,5) {$\bullet$};\node at (5,5) {$\bullet$};
\node at (7,3) {$\bullet$};\node at (7,0) {$\bullet$};
\node at (5,-2) {$\bullet$};\node at (2,-2) {$\bullet$};

\draw (0,0)--(2,5);
\draw (0,0)--(5,5);
\draw (0,0)..controls (2,2) and (7,3)..(7,3);

\draw (0,0)..controls (2,-1.5) and (7,3)..(7,3);
\draw (2,-2)--(7,3);
\draw (2,-2)--(7,0);

\draw[fill=white] (0,0)..controls (2,2) and (2,-1)..(0,0);
\node at (0,0) {$\bullet$};

\draw[blue] (4.5,-2) arc (180:45:0.5);
\draw[blue] (1.5,-1.5) arc (135:35:1);

\draw[blue] (1,4) arc (-135:-110:1.4);
\draw[blue] (4,5) arc (-180:-45:1);

\node at (3.5,-3) {$\Lambda_2=(\Delta_2,d_2) $};

\end{scope}

\end{tikzpicture}}\]

In these two cases,  we are again in case $(c)$. One computes that ${\rm Arf}(d_1)=0$ and ${\rm Arf}(d_2)=0$. Hence $\Lambda_1$ and $\Lambda_2$ are derived equivalent (even if they do not come from the same triangulation), while $\Lambda_0$ and $\Lambda_1$ are not (even if they  come from the same triangulation).

\section*{Acknowledgements}
The first author would like to thank Pierre Dehornoy for interesting discussions on winding numbers and for pointing us to reference \cite{Chillingworth}.
We thank Bernhard Keller and Henning Krause for interesting discussions.
We also thank Gustavo Jasso and Sondre Kvamme for pointing our the necessity of proving Proposition~\ref{prop::maximal-presilting-are-silting}.  Some results in this paper were presented at the ARTA 7 in M\'exico in September 2018 and at the workshop \emph{Stability conditions and representation theory of finite-dimensional algebras} in Oaxaca in October--November 2018.  We would like to thank the organizers of both events.


\bibliographystyle{alpha}
\bibliography{gentleSilting}

\end{document}